\documentclass[letterpaper,10pt,reqno,onefignum,onetabnum]{amsart}
\usepackage[english]{babel}
\usepackage{amsmath}
\usepackage{amsthm}
\usepackage{verbatim}
\usepackage{mathrsfs}
\usepackage{bm}
\usepackage[foot]{amsaddr}
\usepackage[dvipsnames]{xcolor}
\usepackage{hyperref}
\usepackage{cleveref}
\usepackage{algorithm}
\usepackage{algorithmic}
\usepackage{float}
\usepackage{lipsum}
\usepackage{amsfonts}
\usepackage{amssymb}
\usepackage{graphicx}
\usepackage{epstopdf}
\usepackage{cases}
\usepackage{multirow}
\ifpdf
\DeclareGraphicsExtensions{.eps,.pdf,.png,.jpg}
\else
\DeclareGraphicsExtensions{.eps}
\fi
\usepackage{verbatim}
\usepackage{mathrsfs}
\usepackage{bm}

\usepackage{amsthm}
\usepackage{thmtools}
\usepackage{cleveref}
\hypersetup{
	colorlinks = true,
	linkcolor = OliveGreen,
	anchorcolor = OliveGreen,
	citecolor = OliveGreen,
	filecolor = OliveGreen,
	urlcolor = OliveGreen
}

\newtheorem{definition}{Definition}[section]
\newtheorem{teo}{Theorem}[section]
\newtheorem{prop}[teo]{Proposition}
\newtheorem{lem}[teo]{Lemma}
\newtheorem{cor}[teo]{Corollary}

\newtheorem{assumption}[teo]{Assumption}

\newtheorem{rem}[teo]{Remark}

\usepackage{lipsum}
\usepackage{amsfonts}
\usepackage{amssymb}
\usepackage{graphicx}
\usepackage{epstopdf}

\usepackage{cases}
\usepackage{multirow}
\usepackage{tikz}
\usepackage{booktabs}%
\usepackage{hhline}
\usetikzlibrary{matrix}
\usetikzlibrary{arrows}
\ifpdf
\DeclareGraphicsExtensions{.eps,.pdf,.png,.jpg}
\else
\DeclareGraphicsExtensions{.eps}
\fi
\usepackage{verbatim}
\usepackage{mathrsfs}
\usepackage{bm}
\usepackage{color}
\usepackage{xcolor}
\usepackage{caption}
\usepackage{subcaption}
\usepackage{enumitem}
\usepackage{todonotes}
\usepackage{natbib}

\usepackage{geometry}
\numberwithin{equation}{section}
\numberwithin{equation}{section}

\floatname{algorithm}{Line Search}

\DeclareFontEncoding{FMS}{}{}
\DeclareFontSubstitution{FMS}{futm}{m}{n}
\DeclareFontEncoding{FMX}{}{}
\DeclareFontSubstitution{FMX}{futm}{m}{n}
\DeclareSymbolFont{fouriersymbols}{FMS}{futm}{m}{n}
\DeclareSymbolFont{fourierlargesymbols}{FMX}{futm}{m}{n}
\DeclareMathDelimiter{\nr}{\mathord}{fouriersymbols}{152}{fourierlargesymbols}{147}
\DeclareMathOperator*{\argmax}{arg\,max}

\DeclareMathDelimiter{\nr}{\mathord}{fouriersymbols}{152}{fourierlargesymbols}{147}
\DeclareMathAlphabet{\mathpzc}{OT1}{pzc}{m}{it}

\DeclareMathOperator*{\maximize}{maximize}

\def \E{\mathbb{E}}

\def \N{\mathbb{N}}

\def \P{\mathbb{P}}

\def \R{\mathbb{R}}

\def\Fc{{\mathcal F}}

\def\Ic{{\mathcal I}}

\newcommand{\lb}[2]{%
\ifmmode
  \text{\color{red}\textbf{[#1]}}\ {\color{blue}#2}%
\else
  {\color{red}\textbf{[#1]}}\ {\color{blue}#2}%
\fi
}

\title[Multi-Principal Competition for an Exclusive Agent]{ Multi-Principal Competition for an Exclusive Agent: characterization and approximation of Nash Equilibria}

\author{Luis Brice\~no-Arias}
\address{Departamento de Matemática, Universidad Técnica Federico Santa María,
Santiago, Chile}
\email{luis.briceno@usm.cl}

\author{Nicol\'as Hern\'andez-Santib\'a\~nez}
\email{nicolas.hernandezs@usm.cl}

\author{Vicente Moreno-Garrido}
\email{vicente.moreno@usm.cl}

\begin{document}
	\begin{abstract}

This paper introduces and studies a Principal-Agent competition model in which two risk-neutral Principals compete to hire a single Agent. The Agent chooses to work exclusively for one Principal, making the Agent's reservation utility endogenous, as it is determined by the competing offers. We characterize the Nash equilibria of the resulting game by relating them to the individual Principal-Agent contracting problems and analyzing the Agent's selection bias. We further introduce two numerical algorithms to approximate Nash equilibria by locating sign changes of the value functions of the Principals' individual problems. The first is a Newton-type method that exploits differentiability of the value functions, while the second is a bisection method that does not require differentiability. The proposed procedures are illustrated through several examples involving exponential, logarithmic, and constant relative risk aversion (CRRA) utility functions within the polynomial framework of \cite{renner15principal}. We also provide technical results regarding the regularity of the value function of a classical Principal-Agent problem, including differentiability under suitable conditions.

		\par
		\bigskip
		
		
	\end{abstract}
	
	\maketitle

    \textbf{Acknowledgments.} This work was supported by ANID (Chile) under the grants FONDECYT Iniciación 11240944 and FONDECYT Regular 1230257.

\section{Introduction} \label{sec:intro}

\textbf{The Principal-Agent problem.} A Principal-Agent (PA) relationship arises when a first entity (she, the Principal) offers a contract to a second one (he, the Agent) to do some work or action on her behalf. The action performed by the Agent is costly for him and generates an outcome, possibly random, which benefits the Principal. The Principal has imperfect information about the action of the Agent and has to provide incentives in order to align the interests of both parts and make sure the Agent will perform an action that is beneficial for her. The problem of the Principal is to design the contract that maximizes her own benefits, given that the Agent’s satisfaction is sufficiently high. 

The PA problem was introduced in the 1970s as a \emph{static model} 
and studied purely from a game theory point of view, in which the Principal and the Agent play a Stackelberg game. This means mainly that the interaction between both sides is sequential: first, the Principal offers the contract to the Agent and he can accept it or reject it; second, if the Agent accepts the contract, he performs an action which generates the outcome. Generally speaking, the mathematical problem of a \emph{risk-neutral} Principal can be formulated by
\begin{equation}
\label{e:PAmain-riskneutral}\tag{$P$}
\begin{array}{cl} 
\displaystyle\maximize_{(\xi,a)\,\in\,\Xi\times A} & \mathbb{E}^{a}[X-\xi]  \\[0.2cm]
{\scriptstyle \text{subject to}} & \mathbb{E}^{a}[u(\xi)-c(a)] \geq R_0,   \\[0.2cm]
& a \in\displaystyle \argmax_{b \in A} ~ \mathbb{E}^{b}[u(\xi)-c(b)],
\end{array}
\end{equation}
where $\xi$ denotes the contract offered by the Principal, $a$ denotes the optimal action of the Agent, chosen from an admissible set $A$, $X$ is a random variable representing the outcome of the work, $c$ is the cost function of the Agent and $u$ is the utility function of the Agent. Notice that the expectation $\E^a$ represents the fact that the Agent is hired to control the probability distribution of the random variable $X$ and not directly its value. Consequently, the contract $\xi$ has to be understood as a random variable, measurable with respect to $X$. The first constraint in \eqref{e:PAmain-riskneutral}, the \emph{participation constraint}, represents the fact that the expected utility of the Agent under the offered contract must be above a given reservation utility $R_0$ in order to be accepted. The second constraint in \eqref{e:PAmain-riskneutral}, the \emph{incentive compatibility constraint}, represents that the Agent will choose the action which maximizes his own utility since he is economically rational.

The main difficulty in a PA problem stems from the incentive compatibility constraint, which turns the formulation into a bi-level optimization problem. Because no universal methodology exists for solving this class of problems, much of the literature has historically focused on the \emph{first-order approach}, replacing the Agent's global optimization constraint with its first-order optimality condition. For detailed analyses regarding the validity and limitations of this approach, see \cite{mirrlees1976optimal,mirrlees1999theory}, \cite{holmstrom1979moral}, \cite{grossman1992analysis}, \cite{rogerson1985first} and \cite{jewitt1988justifying}. 

To keep track of the data of the Principal-Agent problem, we denote by $V(X,u,c,R_0)$  the value of Problem \eqref{e:PAmain-riskneutral} and by $S(X,u,c,R_0)$ the set of solutions to Problem \eqref{e:PAmain-riskneutral}.

\subsection*{Numerical approach for rational utility functions}

Let us assume that the set of actions of the agent is one-dimensional and compact. Without loss of generality (after rescaling and shifting), suppose $A=[-1,1]$. This case dominates the literature on the first order approach and the applied and computational literature (see \cite{araujo01, Judd98, armstrong10}). In \cite{renner15principal} an algorithm based on a polynomial optimization approach is developed under the assumption that the Agent's utility is a rational function. By using a sum-of-squares of polynomials, an equivalent nonlinear mathematical programming problem is obtained, which is then solved by a global optimization procedure.

Let us detail the main ideas from this approach. Note that the Agent's \emph{incentive compatibility} condition is equivalent to
\begin{equation}
\label{e:red}
\begin{cases}
\mathbb{E}^{b}[u(\xi)-c(b)]\leq \rho,\:\forall b\in A \\
\mathbb{E}^{a}[u(\xi)-c(a)]=\rho,
\end{cases}
\end{equation}
for some $\rho\in\mathbb{R}$. Now, if the expected utility function of the Agent is rational, that is, a quotient of polynomials with the following form 
\begin{equation}
\mathbb{E}^{b}[u(\xi)-c(b)] =\frac{\sum_{i=1}^dc_i(\xi)b^i}{\sum_{i=1}^df_i(\xi)b^i},
\end{equation}
for some $d\in\mathbb{N}$, then \eqref{e:red} is equivalent to
\begin{equation}
\label{e:red2}
\begin{cases}
\rho \sum_{i=1}^df_i(\xi)b^i-\sum_{i=1}^dc_i(\xi)b^i\geq 0,\:\forall b\in A\\
\sum_{i=1}^dc_i(\xi)a^i=\rho \sum_{i=1}^df_i(\xi)a^i.
\end{cases}
\end{equation}
It turns out that the non-negativity of the polynomial $p\colon b\mapsto \rho \sum_{i=1}^df_i(\xi)b^i-\sum_{i=1}^dc_i(\xi)b^i$ of degree $d$ in the first equation in \eqref{e:red2} can be rewritten as a sum of squares of polynomials 
(see, e.g., \cite[Theorem~2.6]{Lasserre10} and \cite[Theorem~3.23]{Laurent09}). Therefore, we have
\begin{equation}
\label{e:eqpol}
(\forall b\in[-1,1])\quad p(b)=\sigma_0(b)+\sigma_1(b) (1-b^2),    
\end{equation}
where $\sigma_0,\sigma_1$ are 
sum of squares, \emph{i.e.}, for every $b\in[0,1]$,
$$\sigma_0(b) =v_n(b)^{\top}Qv_n(b)\quad\text{and}\quad \sigma_1(b) =v_{n-1}(b)^{\top}Rv_{n-1}(b),$$
where $n=[d/2]+1$, for every $k\in\N$,
$v_k(b)=(1,b,\ldots,b^k)$, and $Q=[q_{i,j}]$ and $R=[r_{i,j}]$ are $(n+1)\times (n+1)$ and $n\times n$ symmetric positive semidefinite real matrices, respectively. Then, \eqref{e:PAmain-riskneutral} can be written equivalently as
\begin{equation}
\label{e:relaxAP}
\begin{array}{cl} 
\displaystyle\maximize_{\xi\in\Xi,\, a\in[-1,1],\,\rho\in\mathbb{R},Q,R\succeq 0} & \mathbb{E}^{a}[X-\xi]  \\[0.2cm]
{\scriptstyle \text{subject to}} & 
\sum_{i=1}^dc_i(\xi)a^i-R_0 \sum_{i=1}^df_i(\xi)a^i\geq 0\\[0.2cm]
& \sum_{i=1}^dc_i(\xi)a^i=\rho \sum_{i=1}^df_i(\xi)a^i\\[0.2cm]
&\rho f_0(\xi)-c_0(\xi)=q_{0,0}+r_{0,0}\\[0.2cm]
(\forall k\in\{1,\ldots,d\})\quad& 
\rho f_k(\xi)-c_k(\xi)=\sum_{i+j=k}q_{i,j}+
\sum_{i+j=k}r_{i,j}-\sum_{i+j=k-2}r_{i,j},
\end{array}
\end{equation}
which is a more tractable optimization problem than \eqref{e:PAmain-riskneutral}. In \cite{renner15principal}, Problem \eqref{e:relaxAP} is solved via the globally convergent algorithm Gloptipoly proposed in \cite{Henrion09} and inspired in \cite{Lasserre01}.

\medskip
\subsection*{Two principals in competition}

In many real-life economic transactions, agents receive competing contract offers from multiple employers, clients, or investors. As examples, we can mention key executive hires choosing between rival firms, researchers deciding between incompatible research grants, or specialized suppliers weighing exclusive vendor contracts. In the multi-principal literature, two primary interaction regimes are typically studied based on how the Agent allocates effort: non-exclusive agency, where the Agent accepts contracts from many Principals and distributes effort across different projects, and exclusive agency, where the Agent can accept at most one contract and work exclusively for a single Principal. 

To simplify the exposition, consider two Principals, $P_1$ and $P_2$, each of them attempting to hire the same Agent by offering contracts $\xi_1\in\Xi_1$ and $\xi_2\in\Xi_2$, respectively. In non-exclusive agency, the agent would choose a joint action $(a_1,a_2)\in A_1\times A_2$ which generates an outcome $(X_1,X_2)$ where each coordinate impacts the corresponding Principal. In such a setting, the problem of each principal differs from the classical problem \eqref{e:PAmain-riskneutral}, since the Agent has additional incentives coming from the other contract. They Principals play a game in which both of them provide the reservation utility of the Agent and the action $(a_1,a_2)$ may be legally constrained (\emph{e.g.}, maximum workload). 

In this paper we focus on exclusive agency. We model the case in which the Agent will choose only one of the Principals, let us say $P_i$, and will perform a single action $a_i$ that will generate the outcome $X_i$. Compared to the classic model \eqref{e:PAmain-riskneutral}, the reward functions of the Principals change drastically since one of them will not hire the Agent. In this model, the reservation utility of the Agent becomes endogenous as it is linked to the offers that the two Principals are willing to make. Our goal is to characterize the Nash equilibria of the game in this competition model and, based on \cite{renner15principal}, to propose algorithms to find the equilibrium numerically. We present all our results in the setting with two Principals, although they have natural extensions to the case of more Principals. This choice is made solely to simplify the presentation and notation.


\medskip
\subsection*{Related literature} There are several works that consider more than one principal interacting with a single agent. For instance, \cite{baron1985noncooperative} studies a pollution regulation problem with non-localized externalities. The author analyzes the non-cooperative interaction between two regulators overseeing a monopolistic firm and compares it with the cooperative equilibrium. It is worth noting that the regulators act sequentially with respect to each other, and that their problem involves adverse selection rather than moral hazard with the firm. A more general model is studied in \cite{bernheim1986common}, where many risk-neutral principals hire a common agent who chooses the probabilities with which the outcome takes different (finite) values. The authors study the efficiency and existence of non-cooperative equilibria. \cite{bernheim1985common} show that, in a risk-neutral setting, competing principals without communication may effectively collude when hiring a common agent. In a related contribution, \cite{dixit1997common} develop a general theory of common agency without moral hazard and apply the framework to government policy making. They characterize equilibrium contracts and highlight how coordination issues among principals shape the outcome. \cite{biglaiser1993principals} study a model in which two non-identical principals compete for the exclusive services of an agent, under both moral hazard and adverse selection. They characterize regions of dominance for each of the principals. Multi-market interaction can also arise under a single principal, for example, \cite{braverman1982sharecropping} analyze how a single landlord interlinks credit and land contracts under moral hazard to control the tenant's incentives across multiple domains.

In the case of continuous-time models, the existing literature is relatively scarce. In \cite{mastrolia2018principal}, the authors study the continuous-time multiprincipal-agent problem, assuming that the agent manages a project for each of the principals. The values of the principals are characterized through a system of Hamilton--Jacobi--Bellman (HJB) equations that can be solved when the principals are risk-neutral. In \cite{hu2023principal}, also in continuous time, the switching problem of an agent who works exclusively for one principal at a time is studied. We point out that, in \cite{hu2023principal}, the agent controls only the intensity of a Poisson random measure that drives the switching process.

\subsection*{Overview of the paper} In \Cref{sec:model} we present the competition game between the Principals, a characterization of its Nash equilibria and we propose algorithms for its numerical approximation. In \Cref{sec:numerics}, we present numerical examples including exponential, logarithmic and CRAA utility functions. In \Cref{sec:conclusion} we discuss our results and main contributions. We include auxiliary results, regarding the classical Principal-Agent problem, in the Appendix. 






\section{The competition model and its equilibria} \label{sec:model}

Consider two Principals, $P_1$ and $P_2$, that want to hire the same Agent. We assume that the Agent can be hired by only one of the Principals. Let $X_1$ and $X_2$ be two real-valued independent random variables that represent the outcome of the Agent's work for $P_1$ and $P_2$ respectively. Since the Principals compete for the Agent, we are interested in characterizing the Nash equilibrium of the game they play. The action of the $i-$th Principal consists of a pair $(\xi_i,a_i)\in\Xi_i\times A_i$ of contracts and recommended actions, aimed at maximizing the benefit from $X_i$, whose distribution is controlled by the Agent's action $a_i$. As usual in contract theory, the set $\Xi_i$ is assumed to be some appropriate set of random variables adapted to $X_i$, which means that we can write $\xi_i=w_i(X_i)$ for some measurable function $w_i$.

As already mentioned, we assume that the Agent can only choose one effort, and work for only one of the Principals. Therefore, given a pair of contracts $(\xi_1,\xi_2)$ offered by the Principals, the Agent faces the following problem 
$$
\max \bigg\{ \sup_{a_1\in A_1 }U_1(\xi_1,a_1) , \sup_{a_2\in A_2}U_2(\xi_2,a_2)\bigg\},
$$
where, for $i\in\{1,2\}$, the utility functions are given by $U_i(\xi,a):=\E^{a}[u_i(\xi)-c_i(a)]$, with $u_i:\R\to\R$ and $c_i:A_i\to\R$ being the Agent's utility and cost functions associated to Principal $i$, respectively. We make a distinction on the perception of the Agent when working for the different Principals, in utilities and costs, based on personal preferences. Of course, we can also consider $u_1=u_2$ and $c_1=c_2$. 
For every $i\in\{1,2\}$, we assume that $u_i$ is
strictly increasing, concave (therefore continuous) and differentiable and that $c_i$ is increasing and convex. 

We assume that $A_1$ and $A_2$ are compact subsets of $\R$ and that $\Xi_1\cap\Xi_2$ contains the deterministic contract $0$. Then, for every $i\in\{1,2\}$, we define the following non-empty\footnote{It is easy to see that $(0,\underline{a_i})\in\Sigma_i$ for $\underline{a_i}:=\inf A_i$.} sets
$$
\Sigma_i=\Big\{(\xi,a)\in\Xi_i\times A_i\,:\,a\in\argmax_{b\,\in\,A_i}\, U_i(\xi,b)\Big\}.
$$
The Agent can thus solve independently two optimization problems and compare their values, which naturally defines a partition of the set of joint actions of the Principals
\begin{equation}
\label{e:defC_ii}
\begin{aligned}
C^{0}_i&=\{(\xi_i,a_i,\xi_{-i},a_{-i})\in \Sigma_i\times \Sigma_{-i}:U_i(\xi_i,a_i)=U_{-i}(\xi_{-i},a_{-i})\},\\
C^i_i&=\{(\xi_i,a_i,\xi_{-i},a_{-i})\in \Sigma_i\times \Sigma_{-i}:U_i(\xi_i,a_i)>U_{-i}(\xi_{-i},a_{-i})\},\\
C^{-i}_i&=\{(\xi_i,a_i,\xi_{-i},a_{-i})\in \Sigma_i\times \Sigma_{-i}:U_i(\xi_i,a_i)<U_{-i}(\xi_{-i},a_{-i})\},
\end{aligned}
\end{equation}
where $-i\in\{1,2\}\setminus\{i\}$. In our notation, superscript indicate the Principal who is providing the highest value to the Agent if he accepts the contract (zero indicates equality). Indeed, note that for every $i\in\{1,2\}$ and $(\xi_i,a_i,\xi_{-i},a_{-i})\in \Sigma_i\times \Sigma_{-i}$,
we have 
$$(\xi_i,a_i,\xi_{-i},a_{-i})\in C^0_i\quad\Leftrightarrow\quad (\xi_{-i},a_{-i},\xi_i,a_i)\in C^0_{-i}$$
and 
$$(\xi_i,a_i,\xi_{-i},a_{-i})\in C^i_i\quad\Leftrightarrow\quad (\xi_{-i},a_{-i},\xi_i,a_i)\in C^i_{-i}.$$
The Principals play a game in which the set of actions of Principal $i\in\{1,2\}$ is $\Sigma_i$ and her payoff function is
\begin{align}
\label{e:defJi}
    J_i(X_i,(\xi_i,a_i),(\xi_{-i},a_{-i}),\gamma_i) & =\left\{\begin{array}{ccl}
        \mathbb{E}^{a_i}[X_i-\xi_i], &\text{ if}&(\xi_i,a_i,\xi_{-i},a_{-i})\in C^i_i; \\
        \gamma_i \mathbb{E}^{a_i}[X_i-\xi_i], &\text{ if}&(\xi_i,a_i,\xi_{-i},a_{-i})\in C^0_i;\\
        0, &\text{ if}&(\xi_i,a_i,\xi_{-i},a_{-i})\in C^{-i}_i, \\
    \end{array}\right. 
\end{align}
where $\gamma_i\in[0,1]$ is the Agent's bias towards Principal $i$, that is, the probability of working for her if the two offers provide the same expected utility. We assume $\gamma_1+\gamma_2=1$.

Let us denote by $BR_i(\xi_{-i},a_{-i})$ the best-response set of the $i-$th Principal to the action of the other, and let ${V_i}(\xi_{-i},a_{-i})$ be the corresponding value. Namely
\begin{equation*}
\begin{aligned}
(\forall i\in\{1,2\})\quad V_i(\xi_{-i},a_{-i})&:=&\sup_{(\xi_i,a_i)\in \Sigma_i}&\hspace{0.3cm}J_i(X_i,(\xi_i,a_i),(\xi_{-i},a_{-i}),\gamma_i),
\end{aligned}
\end{equation*}
and $BR_i(\xi_{-i},a_{-i})$ is the, eventually empty, set of maximizers in the definition of $V_i(\xi_{-i},a_{-i})$. We define now the Nash equilibrium of the game played by the Principals.

\begin{definition}
We say that the couple of actions $(\xi_1^\star,a_1^\star),(\xi_2^\star,a_2^\star)$ is a Nash-equilibrium of the game $PPA((X_1,X_2),(u_1,u_2),(c_1,c_2),(\gamma_1,\gamma_2))$ if
\begin{equation*}
        (\xi_1^\star,a_1^\star)\in {BR_1}(\xi_2^\star,a_2^\star), ~ (\xi_2^\star,a_2^\star)\in {BR_2}(\xi_1^\star,a_1^\star).
\end{equation*}
When the context is clear, we simply refer to such a pair as a Nash equilibrium of the PPA game.
\end{definition}

\subsection{Characterization of Nash equilibria}

In this model, from the point of view of the Principals, the reservation utility of the Agent is given endogenously by the offer of the other Principal. In the classical PA problem, we know that the participation constraint of the Agent is satisfied with equality, see  \Cref{lema:matchingR0}. In the next result we show that a necessary condition for a Nash equilibrium is to bind the \emph{participation constraint} in the standard Principal-Agent problem. More precisely, it states that, in a Nash equilibrium, the Agent would receive the same expected utility from any Principal if hired directly. Moreover, none of the Principals will hire the Agent if that results in a loss of utility.


\begin{lem}\label{lema:first-properties}
Let $(\xi_1,a_1),(\xi_2,a_2)$ be a Nash equilibrium of the PPA game.  Then for every $i\in\{1,2\}$
\begin{enumerate}[label=(\roman*)]
    \item\label{lema:first-propertiesi}  $(\xi_i,a_i,\xi_{-i},a_{-i})\in C^0_i$.
    \item\label{lema:first-propertiesii} $J_i(X_i,(\xi_i,a_i),(\xi_{-i},a_{-i}),\gamma_i)\geq0$.
\end{enumerate}
\end{lem}

\begin{proof} Take $i=1$, the other case being analogous. 

\ref{lema:first-propertiesi}: Suppose by contradiction that $(\xi_1,a_1,\xi_2,a_2)\in C^1_1$ (the case $C^2_1$ is analogous). Then, it follows from \eqref{e:defC_ii} that
$$
\Delta:= U_1(\xi_1,a_1)-U_2(\xi_2,a_2)>0.
$$
Let $\varepsilon\in\left]0,\Delta\right[$. By \Cref{lemma:contratos-perturbados}, the contract $\hat{\xi}_1=u_1^{-1}(u_1(\xi_1)-\varepsilon)$ is such that $\hat{\xi}_1<\xi_1$ and 
$$
    a_1\in\argmax_{b\in A_1}U_1(\hat{\xi}_1,b),\quad U_1(\hat{\xi}_1,a_1)=U_1(\xi_1,a_1)-\varepsilon.
$$ 
Therefore $(\hat{\xi}_1,a_1,\xi_2,a_2)\in C^1_1$. Moreover, since $\hat{\xi}_1<\xi_1$, we have
    $$
    J_{1}(X_1,(\hat{\xi}_1,a_1),(\xi_2,a_2),\gamma_1)=\mathbb{E}^{a_1}[X_1-\hat{\xi}_1]
>\mathbb{E}^{a_1}[X_1-\xi_1]
    =J_{1}(X_1,(\xi_1,a_1),(\xi_2,a_2),\gamma),
    $$ 
which contradicts $(\xi_1,a_1)\in BR_1(\xi_2,a_2)$.

\ref{lema:first-propertiesii}: 
For arbitrary $\delta>0$, define the contract $\tilde{\xi}_1:=u_1^{-1}(u_1(\xi_1)-\delta)$. From \ref{lema:first-propertiesi} and \Cref{lemma:contratos-perturbados}.\ref{lemma:contratos-perturbadosi}, we have that $(\tilde{\xi}_1,a_1,\xi_2,a_2)\in C^2_1$ and 
\[
a_1\in\argmax_{b\in A_1}U_1(\tilde{\xi}_1,b).
\]
Since $(\xi_1,a_1)\in BR_1(\xi_2,a_2)$, this implies
$$
J_1(X_1,(\xi_1,a_1),(\xi_2,a_2),\gamma_1)
\geq J_1(X_1,(\tilde{\xi}_1,a_1),(\xi_2,a_2),\gamma_1)=0,
$$
which proves the desired inequality. \qedhere
\end{proof}
In view of \Cref{lema:first-properties}, from now on we will assume that $C^0_1\neq \emptyset$ (and, thus, $C^0_2\neq \emptyset$).  

We now incorporate the role of $\gamma_1$ and $\gamma_2$ to characterize the Nash equilibria of the game.
 The next lemma relates these constants to the sign of the utilities that the Principals would receive via directly contracting the Agent. Note that, as a consequence, their values may rule out the existence of Nash equilibria in the PPA game. This fact has already been mentioned by \cite{biglaiser1993principals}, where the authors assume that an indifferent Agent always works for the Principal with the highest payoff. For the sake of generality, we do not make such an assumption and provide in \Cref{def:consistent-gammas} a consistency property which is clearly satisfied by the aforementioned rule. 
 
 We also prove that in any Nash equilibrium the principals should have different utility signs under the assumption of a strictly positive upper gap given in the following definition. We denote by ${\rm im}\,u=u(\R)$ the image of the real valued function $u\colon\R\to\R$.

 \begin{definition}
(i) For any $i\in\{1,2\}$, let $\xi_i\in\Xi_i$. We define the upper-gap of $\xi_i$ as
\[
\alpha_i(\xi_i):=\sup\{\delta\in\R:u_i(\xi_i)+\delta\leq\sup({\rm im}\,u_i), ~ \P-a.s.\}. 
\]
(ii) We say that $(\xi_1,\xi_2)\in\Xi_1\times\Xi_2$ satisfies \eqref{ass:H} whenever 
\begin{equation}\label{ass:H}\tag{$H$}
(\forall i\in\{1,2\}) \quad \alpha_i(\xi_i)>0.
\end{equation}
\end{definition}

\begin{rem} \label{rem:when-assump-H-holds}
Note that condition \eqref{ass:H} is trivially satisfied if each $u_i$ is surjective or each $\xi_i$ is bounded from above (\emph{e.g.}, when $X_i$ have values in a finite set or $\Xi_i$ contains only bounded random variables). 
\end{rem}


\begin{lem}\label{lema:gamma-01}
Suppose that $(\xi_1,a_1),(\xi_2,a_2)$ be a Nash equilibrium of the PPA game. 
\begin{enumerate}[label=(\roman*)]
    \item\label{lema:gamma-01i} For every $i\in\{1,2\}$, if $\mathbb{E}^{a_i}[X_i-\xi_i]<0$, then $\gamma_i=0$.
    \item\label{lema:gamma-01ii} For every $i\in\{1,2\}$, if $\mathbb{E}^{a_i}[X_i-\xi_i]>0$ and $\alpha_i(\xi_i)>0$, then $\gamma_i=1$.
    \item\label{lema:gamma-01iii} If $(\xi_1,\xi_2)$ satisfies \eqref{ass:H}, there exists $i\in\{1,2\}$ such that $\mathbb{E}^{a_i}[X_i-\xi_i]\geq0$, $\mathbb{E}^{a_{-i}}[X_{-i}-\xi_{-i}]\leq0$.
\end{enumerate}
\end{lem}
\begin{proof} The proof is analogous for each $i\in\{1,2\}$, so we provide a proof for $i=1$. By \Cref{lema:first-properties}\ref{lema:first-propertiesi}, we have $(\xi_1,a_1,\xi_2,a_2)\in C^0_1$ and $J_1(X_1,(\xi_1,a_1),(\xi_2,a_2),\gamma_1)=\gamma_1\mathbb{E}^{a_1}[X_1-\xi_1]\geq0$.

\ref{lema:gamma-01i}: If $\mathbb{E}^{a_1}[X_1-\xi_1]<0$, then it is clear that $\gamma_1=0$. 

\ref{lema:gamma-01ii}: By contradiction, suppose $\mathbb{E}^{a_1}[X_1-\xi_1]>0$, $\alpha_1(\xi_1)>0$ and $\gamma_1<1$. Therefore, we have $(1-\gamma_1)\mathbb{E}^{a_1}[X_1-\xi_1]>0$. 
It follows from \Cref{lemma:contratos-perturbados}\ref{lemma:contratos-perturbadosiii}
that, for $0<\varepsilon<(1-\gamma_1)\mathbb{E}^{a_1}[X_1-\xi_1]$,
there exists $\delta\in \left]0,\alpha_1(\xi_1)\right[$ such that the contract $\hat{\xi}_1=u_1^{-1}(u_1(\xi_1)+\delta)$ satisfies 
$$
0<\mathbb{E}^{a_1}[\hat{\xi}_1-\xi_1]\le\varepsilon,
$$
and from \Cref{lemma:contratos-perturbados}\ref{lemma:contratos-perturbadosi}, we have $(\hat{\xi}_1,a_1,\xi_2,a_2)\in C^1_1$.
Then, we conclude
$$
\begin{aligned} J_1(X_1,(\hat{\xi}_1,a_1),(\xi_2,a_2),\gamma_1)&=\mathbb{E}^{a_1}[X_1-\hat{\xi}_1]\\
&=\mathbb{E}^{a_1}[X_1-\xi_1]+\mathbb{E}^{a_1}[\xi_1-\hat{\xi}_1]\\
&>\mathbb{E}^{a_1}[X_1-\xi_1]-(1-\gamma_1)\mathbb{E}^{a_1}[X_1-\xi_1]\\
&=\gamma_1\mathbb{E}^{a_1}[X_1-\xi_1]\\
&=J_1(X_1,(\xi_1,a_1),(\xi_2,a_2),\gamma_1),
\end{aligned}
$$
which is a contradiction with $(\xi_1,a_1)\in BR_1(\xi_2,a_2)$. 

\ref{lema:gamma-01iii}:
Suppose by contradiction that $\mathbb{E}^{a_1}[X_1-\xi_1]$ and $\mathbb{E}^{a_2}[X_2-\xi_2]$ are both strictly positive or both strictly negative. If both are strictly positive, then \eqref{ass:H} and \Cref{lema:gamma-01}\ref{lema:gamma-01ii} imply that $\gamma_1=1=\gamma_2$, which contradicts $\gamma_1+\gamma_2=1$. The argument is similar if both are strictly negative, since \Cref{lema:gamma-01}\ref{lema:gamma-01i} would imply $\gamma_1=0=\gamma_2$.
\qedhere
\end{proof}

Since the previous result gives us a necessary relation between the values of $\gamma_1$ and $\gamma_2$ and 
a Nash equilibrium, we provide the following definition to characterize equilibria.

\begin{definition}\label{def:consistent-gammas}
Given $(\xi_1,a_1,\xi_2,a_2)\in C_1^0$, we say that
$(\gamma_1,\gamma_2)$ is consistent with $(\xi_1,a_1,\xi_2,a_2)$ if 
$$(\forall i\in\{1,2\})\quad 
\begin{cases}
\mathbb{E}^{a_i}[X_i-\xi_i]<0\:\:\Rightarrow\:\:\gamma_i=0,\\
\mathbb{E}^{a_i}[X_i-\xi_i]>0\:\:\Rightarrow\:\:\gamma_i=1.
\end{cases}
$$    
\end{definition}

It is clear from \Cref{lema:gamma-01} that $(\gamma_1,\gamma_2)$ has to be consistent with any Nash equilibrium. Moreover, we know from \Cref{lema:gamma-01}\ref{lema:gamma-01iii} that in every Nash equilibrium, one (and only one) of the Principals would make a profit (eventually zero) if hiring the agent directly. The following result provides a characterization of the best-response strategies of a Principal $i\in\{1,2\}$ through her individual contracting problem. Firstly, the Principal who can make a profit will actually play (in the Principal's game) a strategy that solves her individual contracting problem. Secondly, the Principal who would make a loss from hiring the agent has no chance of making a profit since, for the corresponding endogenous reservation utility, her individual contracting problem has a negative value. The second property holds under a slightly stronger condition that we introduce next.


\begin{definition}
Let $i\in\{1,2\}$. We say that $\ell\in\R$ satisfies \eqref{ass:HH} whenever 
\begin{equation}\label{ass:HH}\tag{$H^\prime_i$}
\forall (\eta_i,\beta_i)\in\Sigma_i \text{ s.t. } U_i(\eta_i,\beta_i)=\ell, \quad \alpha_i(\eta_i)>0.
\end{equation}
\end{definition}

\begin{rem}
Note that Condition \eqref{ass:HH} is also satisfied in the cases described in \Cref{rem:when-assump-H-holds}.    
\end{rem}

\begin{prop} \label{prop:equiv-BR}
Let $(\xi_1,a_1,\xi_2,a_2)\in C^0_1$ and suppose that $(\gamma_1,\gamma_2)$ is consistent with $(\xi_1,a_1,\xi_2,a_2)$. Define $R_0\in\R$ as the common expected utility of the Agent, i.e., 
    $$
    R_0:=U_1(\xi_1,a_1)=U_2(\xi_2,a_2).
    $$
 Then, for every $i\in\{1,2\}$, the following hold:
\begin{enumerate}[label=(\roman*)]

\item\label{prop:equiv-BRi} Suppose $\mathbb{E}^{a_i}[X_i-\xi_i]> 0$.
Then, 
 $(\xi_i,a_i)\in BR_i(\xi_{-i},a_{-i})$
 if and only if 
 $(\xi_i,a_i)\in S(X_i,u_i,c_i,R_0)$.

\item\label{prop:equiv-BRii} Suppose $\mathbb{E}^{a_i}[X_i-\xi_i]\le 0$.
Then, the following hold:
\begin{enumerate}
    \item\label{prop:equiv-BRiia} If $V(X_i,u_i,c_i,R_0)\le 0$, then $(\xi_i,a_i)\in BR_i(\xi_{-i},a_{-i})$.
    \item\label{prop:equiv-BRiib} If $R_0$ satisfies \eqref{ass:HH} and $(\xi_i,a_i)\in BR_i(\xi_{-i},a_{-i})$, then $V(X_i,u_i,c_i,R_0)\le 0$.
\end{enumerate}

\end{enumerate}
\end{prop}

\begin{proof}
Define the sets
$$(\forall i\in\{1,2\})(\forall j\in\{0,i,-i\})\quad 
\Gamma^j_i(\xi_{-i},a_{-i}):=\{(\eta_i,\beta_i)\in\Sigma_i\,:\,(\eta_i,\beta_i,\xi_{-i},a_{-i})\in C^j_i\}.
$$
Note that, for every $i\in\{1,2\}$, $\Sigma_i=\Gamma^0_i(\xi_{-i},a_{-i})\cup \Gamma^i_i(\xi_{-i},a_{-i})\cup \Gamma^{-i}_i(\xi_{-i},a_{-i})$
and also
\begin{equation}
\label{e:Fequiv}
(\forall (\eta_i,\beta_i)\in \Sigma_i)\quad (\eta_i,\beta_i)\in \Gamma^0_i(\xi_{-i},a_{-i})\cup \Gamma^i_i(\xi_{-i},a_{-i})\:\:\Leftrightarrow\:\:U_i(\eta_i,\beta_i)\geq R_0,
\end{equation}
which yields
\begin{equation}
\label{e:Vequiv}
V(X_i,u_i,c_i,R_0)\le 0 ~ \Leftrightarrow ~ \forall (\eta_i,\beta_i)\in\Gamma^0_i(\xi_{-i},a_{-i})\cup \Gamma^i_i(\xi_{-i},a_{-i}):~ \mathbb{E}^{\beta_i}[X_i-\eta_i]\le 0,
\end{equation}
and
\begin{equation}
\label{e:Sequiv}
(\xi_i,a_i)\in S(X_i,u_i,c_i,R_0) ~ \Leftrightarrow ~ \forall (\eta_i,\beta_i)\in\Gamma^0_i(\xi_{-i},a_{-i})\cup \Gamma^i_i(\xi_{-i},a_{-i}): \mathbb{E}^{\beta_i}[X_i-\eta_i]\le \mathbb{E}^{a_i}[X_i-\xi_i].
\end{equation}
In addition,  
\eqref{e:defJi} and \Cref{lema:first-properties}\ref{lema:first-propertiesii} imply, for every $i\in\{1,2\}$,
$$\forall(\eta_i,\beta_i)\in \Gamma^{-i}_i(\xi_{-i},a_{-i}): ~ J_i(X_i,(\eta_i,\beta_i),(\xi_{-i},a_{-i}),\gamma_i)=0\le J_i(X_i,(\xi_i,a_i),(\xi_{-i},a_{-i}),\gamma_i),$$
which yields
\begin{equation}
\label{e:BRequiv}
(\xi_i,a_i)\in BR_i(\xi_{-i},a_{-i})\quad\Leftrightarrow\quad 
\begin{cases}
\forall (\eta_i,\beta_i)\in\Gamma^0_i(\xi_{-i},a_{-i}):~ \gamma_i\mathbb{E}^{\beta_i}[X_i-\eta_i]\le\gamma_i\mathbb{E}^{a_i}[X_i-\xi_i]\:\:\text{and}\\
\forall (\eta_i,\beta_i)\in\Gamma^i_i(\xi_{-i},a_{-i}):~ \mathbb{E}^{\beta_i}[X_i-\eta_i]\le\gamma_i\mathbb{E}^{a_i}[X_i-\xi_i].
\end{cases}
\end{equation}

\ref{prop:equiv-BRi}:
Assume $\mathbb{E}^{a_i}[X_i-\xi_i]>0$. From the consistency property we have $\gamma_i=1$, and the equivalence is clear in this case since $(\xi_i,a_i)\in S(X_i,u_i,c_i,R_0)$ is equivalent to the right side of \eqref{e:BRequiv} because of \eqref{e:Sequiv}.

\ref{prop:equiv-BRii}:
Assume $\mathbb{E}^{a_i}[X_i-\xi_i]\le 0$. From the consistency property, note that $\gamma_i\mathbb{E}^{a_i}[X_i-\xi_i]=0$. Indeed, if $\mathbb{E}^{a_i}[X_i-\xi_i]<0$ then $\gamma_i=0$.  \ref{prop:equiv-BRiia}: Since $V(X_i,u_i,c_i,R_0)\le 0$, \eqref{e:BRequiv} follows from \eqref{e:Vequiv}.
\ref{prop:equiv-BRiib}: It is enough to prove that, for every $i\in\{1,2\}$, 
    \begin{equation}
    \label{e:impl}
    \begin{cases}
\gamma_i\mathbb{E}^{a_i}[X_i-\xi_i]=0\:\:\text{and}\\(\xi_i,a_i)\in BR_i(\xi_{-i},a_{-i})  
    \end{cases}
\quad\Rightarrow\quad (\forall (\eta_i,\beta_i)\in \Gamma^0_i(\xi_{-i},a_{-i})\cup \Gamma^i_i(\xi_{-i},a_{-i}))\quad 
 \mathbb{E}^{\beta_i}[X_i-\eta_i]\le 0.  
    \end{equation}
Indeed, let $(\eta_i,\beta_i)\in \Gamma^0_i(\xi_{-i},a_{-i})\cup \Gamma^i_i(\xi_{-i},a_{-i})$. If 
$(\eta_i,\beta_i)\in \Gamma^i_i(\xi_{-i},a_{-i})$, we have
$$\mathbb{E}^{\beta_i}[X_i-\eta_i]=J_i(X_i,(\eta_i,\beta_i),(\xi_{-i},a_{-i}),\gamma_i)\le 
J_i(X_i,(\xi_i,a_i),(\xi_{-i},a_{-i}),\gamma_i)=\gamma_i\mathbb{E}^{a_i}[X_i-\xi_i]=0.$$
Otherwise, suppose that $(\eta_i,\beta_i)\in \Gamma^0_i(\xi_{-i},a_{-i})$. If
$\gamma_i>0$, then $$\gamma_i\mathbb{E}^{\beta_i}[X_i-\eta_i]=J_i(X_i,(\eta_i,\beta_i),(\xi_{-i},a_{-i}),\gamma_i)\le 
J_i(X_i,(\xi_i,a_i),(\xi_{-i},a_{-i}),\gamma_i)=\gamma_i\mathbb{E}^{a_i}[X_i-\xi_i]=0,$$
which yields $\mathbb{E}^{\beta_i}[X_i-\eta_i]\le 0$.
Otherwise, if $\gamma_i=0$, suppose that $\mathbb{E}^{\beta_i}[X_i-\eta_i]>0$ and let $0<\varepsilon<\mathbb{E}^{\beta_i}[X_i-\eta_i]$. Since $\alpha_i(\eta_i)>0$,  Lemma~\ref{lemma:contratos-perturbados}\ref{lemma:contratos-perturbadosi}\&\ref{lemma:contratos-perturbadosiii} imply that there exists $\delta>0$ such that 
$\hat{\eta}_i=u_i^{-1}(u_i(\eta_i)+\delta)$ and $0<\mathbb{E}^{\beta_i}[\hat{\eta}_i-\eta_i]<\varepsilon$. Then $(\hat{\eta}_i,\beta_i,\xi_{-i},a_{-i})\in \Gamma^{i}_i(\xi_{-i},a_{-i})$
and
$$
\mathbb{E}^{\beta_i}[X_i-\hat{\eta}_i]=
\mathbb{E}^{\beta_i}[X_i-\eta_i]+\mathbb{E}^{\beta_i}[\eta_i-\hat{\eta}_i]\ge
\mathbb{E}^{\beta_i}[X_i-\eta_i]-\varepsilon>0.
$$
Hence, $$
J_i(X_i,(\hat{\eta}_i,\beta_i),(\xi_{-i},a_{-i}),\gamma_i)=\mathbb{E}^{\beta_i}[X_i-\hat{\eta}_i]>0=J_i(X_i,(\xi_i,a_i),(\xi_{-i},a_{-i}),\gamma_i),
$$
which contradicts $(\xi_i,a_i)\in BR_i(\xi_{-i},a_{-i})$. Then, $\mathbb{E}^{\beta_i}[X_i-\eta_i]\le 0$ and the result
follows from \eqref{e:Vequiv}.
\qedhere
\end{proof}

We are now able to show a full characterization of all possible Nash equilibria for this problem involving contracts with strictly positive upper-gap.


\begin{teo}\label{thrm:PPA-valuesign}
    Let $(\xi_1,a_1,\xi_2,a_2)\in C^0_1$ and define $R_0\in\R$ as the common expected utility of the Agent, i.e., 
    $$
    R_0:=U_1(\xi_1,a_1)=U_2(\xi_2,a_2).
    $$
    Suppose that $R_0$ satisfies \eqref{ass:HH} for every $i\in\{1,2\}$. 
   Then the following statements are equivalent:
   \begin{enumerate}[label=(\roman*)]
       \item\label{thrm:PPA-valuesigni} $(\xi_1,a_1),(\xi_2,a_2)$ is Nash equilibrium of the PPA game.
       \item\label{thrm:PPA-valuesignii}  $(\gamma_1,\gamma_2)$ is consistent with $(\xi_1,a_1,\xi_2,a_2)$,
       there exists $i\in\{1,2\}$ such that 
   $\mathbb{E}^{a_i}[X_i-\xi_i]\ge0$, $(\xi_i,a_i)\in S(X_i,u_i,c_i,R_0)$ and one of the following holds:
   \begin{enumerate}
       \item $\mathbb{E}^{a_{-i}}[X_{-i}-\xi_{-i}]=0$ and $(\xi_{-i},a_{-i})\in S(X_{-i},u_{-i},c_{-i},R_0)$,
       \item $\mathbb{E}^{a_{-i}}[X_{-i}-\xi_{-i}]<0$ and $V(X_{-i},u_{-i},c_{-i},R_0)\leq0$.
   \end{enumerate}
   \end{enumerate}        
\end{teo}

\begin{proof}
Note that in both implications $(\gamma_1,\gamma_2)$ is consistent with $(\xi_1,a_1,\xi_2,a_2)$. Let $i\in\{1,2\}$ be such that $\mathbb{E}^{a_i}[X_i-\xi_i]\ge0$
and $\mathbb{E}^{a_{-i}}[X_{-i}-\xi_{-i}]\le 0$.
Note that \Cref{prop:equiv-BR}\ref{prop:equiv-BRi}
asserts that
$$(\xi_i,a_i)\in S(X_i,u_i,c_i,R_0)\quad\Leftrightarrow\quad(\xi_i,a_i)\in BR_i(\xi_{-i},a_{-i}).$$
Similarly, if 
$\mathbb{E}^{a_{-i}}[X_{-i}-\xi_{-i}]=0$, we deduce from \Cref{prop:equiv-BR}\ref{prop:equiv-BRi}  that 
$$(\xi_{-i},a_{-i})\in S(X_{-i},u_{-i},c_{-i},R_0)\quad\Leftrightarrow\quad(\xi_{-i},a_{-i})\in BR_{-i}(\xi_{i},a_{i}).$$ 
On the other hand, if 
$\mathbb{E}^{a_{-i}}[X_{-i}-\xi_{-i}]<0$, from \Cref{prop:equiv-BR}\ref{prop:equiv-BRiia} we obtain 
$$V(X_{-i},u_{-i},c_{-i},R_0)\leq0\quad\Rightarrow\quad(\xi_{-i},a_{-i})\in BR_{-i}(\xi_{i},a_{i}).$$
This proves sufficiency. Since $R_0$ satisfies \eqref{ass:HH} for every $i\in\{1,2\}$, from \Cref{lema:gamma-01}\ref{lema:gamma-01iii} we deduce the existence of $i\in\{1,2\}$ such that $\mathbb{E}^{a_i}[X_i-\xi_i]\ge0$ and  $\mathbb{E}^{a_{-i}}[X_{-i}-\xi_{-i}]\le0$. Note that \Cref{prop:equiv-BR}\ref{prop:equiv-BRiib} implies
$$(\xi_{-i},a_{-i})\in BR_{-i}(\xi_{i},a_{i})\quad\Rightarrow\quad V(X_{-i},u_{-i},c_{-i},R_0)\leq0.$$
Hence, we conclude from \Cref{prop:equiv-BR}\ that $(\xi_i,a_i)\in S(X_i,u_i,c_i,R_0)$. This proves necessity.
\qedhere
\end{proof}

\begin{rem} \label{rem:sufficiency-Nash-eq} Sufficiency in the preceding theorem does not require \eqref{ass:HH} for any $i\in\{1,2\}$. Hence, even if  \eqref{ass:HH} is not satisfied, Nash equilibria can be found by verifying the conditions in Theorem~\ref{thrm:PPA-valuesign}\ref{thrm:PPA-valuesignii}, if such type of equilibrium exists. These conditions motivate an algorithm for solving the PPA game, which will be explored in the following sections.
\end{rem}

\subsection{On the existence of equilibria}

\Cref{thrm:PPA-valuesign} provides a characterization of the Nash equilibria of the PPA game. Based on it, we can establish an existence result which relies on being able to solve the individual contracting problems of both Principals. To this end, we define the (possible infinite) highest expected utility of the Agent when working for each Principal
\[
\forall i\in\{1,2\}, \quad \bar{R_i}:= \sup_{(\xi,a)\in\Sigma_i} U_i(\xi,a),
\]
and we consider the following assumption, which is illustrated via some examples in \Cref{sec:existence}.
\begin{assumption} \label{ass:existence-nash-competition} ~
For every $R\leq \min\{\bar{R}_1,\bar{R}_2\}$ and $i\in\{1,2\}$, we have $S(X_i,u_i,c_i,R)\neq \emptyset$.
 \end{assumption}

We present our main existence result, which provides a sufficient condition for the existence of Nash equilibria, based on the value functions of the individual contracting problems of the Principals.

 \begin{prop}\label{prop:sufficient-condition-Nash}
 Suppose that \Cref{ass:existence-nash-competition} holds and that there exists $R_0\le\min\{\bar{R}_1,\bar{R}_2\}$ such that one of the following assertions holds:  
\begin{enumerate}[label=(\roman*)]
\item There exist $i\in\{1,2\}$ such that $\gamma_i =1$ and
         \[
 V(X_i,u_i,c_i,R_0) > 0 \geq V(X_{-i},u_{-i},c_{-i},R_0).
        \] 
\item There exist $i\in\{1,2\}$ such that $\gamma_i =1$ and
         \[
 V(X_i,u_i,c_i,R_0) \geq 0 > V(X_{-i},u_{-i},c_{-i},R_0).
        \] 
\item $V(X_1,u_1,c_1,R_0)=V(X_{2},u_{2},c_{2},R_0)=0$.

\end{enumerate}
Then there exists a Nash equilibrium $(\xi_1,a_1,\xi_2,a_2)\in\Sigma_1\times\Sigma_2$ of the PPA game, satisfying
\begin{equation}
\label{e:sat}
V(X_i,u_i,c_i,R_0) \geq 0 \quad\Longrightarrow \quad (\xi_i,a_i)\in S(X_i,u_i,c_i,R_0).
\end{equation}
 \end{prop}

\begin{proof}
Let $R_0\le\min\{\bar{R}_1,\bar{R}_2\}$, $(\xi_i,a_i)\in S(X_i,u_i,c_i,R_0)$ and $(\xi_{-i},a_{-i})\in S(X_{-i},u_{-i},c_{-i},R_0)$, in view of Assumption~\ref{ass:existence-nash-competition}. It follows from \Cref{lema:matchingR0} that $U_i(\xi_i,a_i)=U_{-i}(\xi_{-i},a_{-i})=R_0$. Moreover, if $E^{a_i}[X_i-\xi_i]>0$, it follows from \Cref{prop:equiv-BR}\ref{prop:equiv-BRi} that $(\xi_i,a_i)\in BR_i(\xi_{-i},a_{-i})$ and, if $E^{a_i}[X_i-\xi_i]=0$, the same result follows from \Cref{prop:equiv-BR}\ref{prop:equiv-BRiia}. Since $V(X_{-i},u_{-i},c_{-i},R_0)\leq 0$, \Cref{prop:equiv-BR}\ref{prop:equiv-BRiia} implies that $(\xi_{-i},a_{-i})\in BR_{-i}(\xi_{i},a_{i})$. Hence, the tuple $(\xi_i,a_i,\xi_{-i},a_{-i})$ is a Nash equilibrium of the PPA game and \eqref{e:sat} holds. 
\end{proof}

\begin{rem} \label{rem:non-optimality-losing-principal}
In the proof of \Cref{prop:sufficient-condition-Nash}, if conditions (i) or (ii) holds, we can choose any $(\xi_{-i},a_{-i})\in\Sigma_{-i}$ such that $(\xi_i,a_i,\xi_{-i},a_{-i})\in C_i^0$. This means the class of Nash equilibria is broader, as it includes actions that are not solutions to the individual contracting problem of Principal $-i$. To simplify the proof, we selected strategies, for both players, that solve their individual contracting problem. However, for the Principal with negative value (and zero bias), any strategy that does not surpass the offer of the other is a best-response.       
\end{rem}

The importance of \Cref{prop:sufficient-condition-Nash} is that it allows to find Nash equilibria of the PPA game by looking at the individual contracting problems of each Principal separately, and comparing their value functions. This feature is what motivates Section~\ref{subsec:algorithm}, in which we propose an algorithm for finding the equilibria of the game played by the Principals.

\subsection{Computation of Nash equilibria} \label{subsec:algorithm}

Based on \Cref{prop:sufficient-condition-Nash} and \Cref{thrm:PPA-valuesign}, we provide numerical procedures to approximate a Nash equilibrium of the PPA game when the value functions of the Principals satisfy \Cref{ass:existence-nash-competition}. 

Note first that, for every $R>\bar{R}:=\min\{\bar{R}_1,\bar{R}_2\}$, the sets $C_1^0=C_2^0=\emptyset$ and, in view of \Cref{lema:first-properties}, there is no Nash equilibrium providing the Agent such level of utility. Hence, we restrict attention to $R\leq\bar R$. We know by \Cref{lemma:decrecimiento-V} that each $V(X_j,u_j,c_j,\cdot)$ is non-increasing for all $j\in\{1,2\}$. In view of \Cref{lema:gamma-01} and \Cref{thrm:PPA-valuesign}, if both $V(X_1,u_1,c_1,R)$ and $V(X_2,u_2,c_2,R)$ are strictly negative or strictly positive for every $R\le \bar{R}$, then there is no Nash equilibrium. 

We therefore focus on the case in which both value functions $V(X_1,u_1,c_1,\cdot)$ and $V(X_2,u_2,c_2,\cdot)$ change sign on $\left]-\infty,\bar R\right]$. Our algorithm constructs the largest interval $\left[R_0^-,R_0^+\right]$ with $R_0^+\le \bar{R}$ such that, for some fixed $i\in{1,2}$,
\begin{equation}\label{eq:signos-principales}
V(X_i,u_i,c_i,R)\geq 0\geq V(X_{-i},u_{-i},c_{-i},R), \quad \forall R\in\left[R_0^-,R_0^+\right].    
\end{equation}

For every $j\in\{1,2\}$, let $R_j\in\R$ denote a sign-change threshold of $V(X_j,u_j,c_j,\cdot)$, namely a value satisfying
$$
(\forall j\in\{1,2\})\quad V(X_j,u_j,c_j,R)=\left\{\begin{array}{cl}
    \geq0, &\text{if }R\in \left]-\infty,R_j\right[,  \\[2pt]
    \leq0, &\text{if }R\in \left]R_j, \bar{R}_j\right].
\end{array}\right.
$$
Then, we can choose $R_0^-:=\min\{R_1,R_2\}$ and $R_0^+:=\max\{R_1,R_2\}$, except in cases as the one depicted in Figure \ref{fig:bad-case}, in which the following two conditions hold simultaneously:
\begin{itemize}
    \item[(1)] $\operatorname{sign}(V(X_i,u_i,c_i,\cdot))\equiv \operatorname{sign}(V(X_{-i},u_{-i},c_{-i},\cdot))$,
    \item[(2)] $0$ does not belong to the image of either $V(X_i,u_i,c_i,\cdot)$ or $V(X_{-i},u_{-i},c_{-i},\cdot)$. 
\end{itemize}

\begin{figure}[H]
    \centering
    \includegraphics[width=0.5\linewidth]{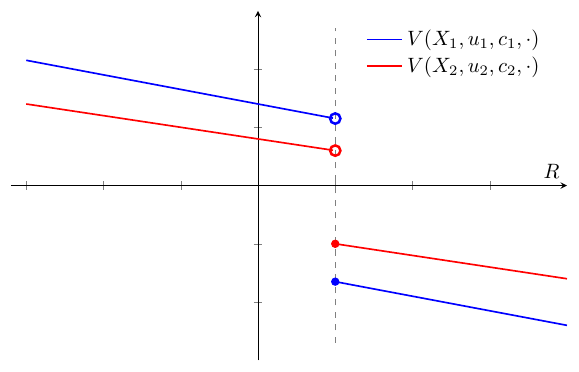}
    \caption{An example of a problematic case}
    \label{fig:bad-case}
\end{figure}

Under the existence of sign-change thresholds, condition (1) fails whenever $R_1\neq R_{2}$. Moreover, condition (2) fails if either value function is continuous. In \Cref{lemma:continuidad-V}, we provide conditions ensuring continuity of the Principals' value functions.

Given its unlikelihood, particularly when the Principals are heterogeneous, we assume that such a problematic case does not occur. Accordingly, we provide two numerical methods to approximate the sign-change thresholds $R_1$ and $R_2$, which determine $R_0^-=\min\{R_1,R_2\}$ and $R_0^+=\max\{R_1,R_2\}$ such that \eqref{eq:signos-principales} holds. The first method is bisection, which exploits the monotonicity\footnote{See \Cref{lemma:decrecimiento-V} for the details.} of $V(X,u,c,\cdot)$ and can be applied even when the value function is discontinuous. The second method is Newton's method, valid under the differentiability of the function $V(X,u,c,\cdot)$, studied in Appendix~\ref{sec:differentiability}. Note however that the function is differentiable almost everywhere due to is monotonicity.


\medskip
\noindent\textbf{Bisection procedure for the sign-change threshold of $V(X,u,c,\cdot)$.} Given reservation utilities $r<R$ satisfying $V(X,u,c,r)>0>V(X,u,c,R)$, and a precision tolerance $\varepsilon_{\mathrm{tol}}>0$.

\begin{align*}\label{e:algotseng}
&\begin{array}{l}
\text{Function }B[X,u,c,R,r,\varepsilon_{\mathrm{tol}}]\\
\left\lfloor
\begin{array}{l}
k=0,~ R_0=R,~ r_0=r\\
\text{While } R_k-r_k>\varepsilon_{\mathrm{tol}}\\
\left\lfloor
\begin{array}{l}
\ell_k=\frac12(R_k+r_k), \\
V_k=V(X,u,c,\ell_k), \\
\text{If }V_k<0, 
\text{then } 
r_{k+1}=r_k,~R_{k+1}=\ell_k,\\
\text{Else } 
r_{k+1}=\ell_k,~ R_{k+1}=R_k,\\
k=k+1, 
\end{array}
\right.\\
\text{Return }\ell_k.
\end{array}
\right.
\end{array}
\end{align*}

\noindent\textbf{Newton's method for the zeros of $V(X,u,c,\cdot)$.}
Given a value $\ell_k\in\R$, if $(\xi_k,a_k)\in S(X,u,c,\ell_k)$, then under the assumptions of \Cref{prop:differentibility-iff},
$V'(X,u,c,\ell_k)=-\E^{a_k}[1/u'(\xi_k)]$. This suggests the Newton iteration
\begin{equation}\label{e:newton-step}
\ell_{k+1}=\ell_k-\frac{V(X,u,c,\ell_k)}{V'(X,u,c,\ell_k)}
=\ell_k+\frac{V(X,u,c,\ell_k)}{\E^{a_k}\!\left[1/u'(\xi_k)\right]},
\end{equation}
safeguarded by the sign-changing bracket $[r_k,R_k]$ (with $V(X,u,c,r_k)>0>V(X,u,c,R_k)$). Whenever the Newton
iterate falls outside $\left]r_k,R_k\right[$, the midpoint is used instead, and the bracket is updated with the sign of
$V(X,u,c,\ell_{k+1})$ exactly as in Function $B$.
\begin{align*}\label{e:algonewton}
&\begin{array}{l}
\text{Function }N[X,u,c,R,r,\varepsilon_{tol},k_{max}]\\
\left\lfloor
\begin{array}{l}
r_0=r,\;R_0=R,\;\ell_0\in\{r,R\}\text{ with smaller }|V(X,u,c,\cdot)|,\;(\xi_0,a_0)\in S(X,u,c,\ell_0)\\
\text{For } k=0,1,\ldots,k_{max}\\
\left\lfloor
\begin{array}{l}
\ell_{k+1}=\ell_k+V(X,u,c,\ell_k)\big/\E^{a_k}[1/u'(\xi_k)]\\
\text{If } \ell_{k+1}\notin\left]r_k,R_k\right[\text{, then } \ell_{k+1}=\tfrac12(r_k+R_k)\\
\text{Solve } (\xi_{k+1},a_{k+1})\in S(X,u,c,\ell_{k+1}) \text{ and set } V_{k+1}=V(X,u,c,\ell_{k+1})\\
\text{If } V_{k+1}>0 \text{ then } r_{k+1}=\ell_{k+1},\,R_{k+1}=R_k \text{ else } r_{k+1}=r_k,\,R_{k+1}=\ell_{k+1}\\
\text{If } |V_{k+1}|\le\varepsilon_{tol} \text{ or } R_{k+1}-r_{k+1}\le\varepsilon_{tol}\text{, then STOP}
\end{array}\right.\\
\text{Return }\ell_{k+1}
\end{array}\right.
\end{array}
\end{align*}
Because of the safeguard, the bracket length is at most halved every two iterations in the worst case, so Function $N$
never performs worse than bisection by more than a constant factor, while, near a zero at which $V(X,u,c,\cdot)$ is twice
continuously differentiable with non-vanishing derivative, the iteration \eqref{e:newton-step} converges quadratically.

Now, given a precision tolerance $\varepsilon_{tol}$ and, for every $i\in\{1,2\}$, two values $r_i<R_i$ such that $V(X_i,u_i,c_i,r_i)>0>V(X_i,u_i,c_i,R_i)$, to compute the Nash equilibria, we  propose, 

\begin{enumerate}
\item For every $i\in\{1,2\}$, set $k^i_{tol}=\ln_2((R_i-r_i)/\varepsilon_{tol})$.
\item Compute $\ell_i$ via $B[X_i,u_i,c_i,R_i,r_i,k^i_{tol}]$ or $N[X_i,u_i,c_i,R_i,r_i,\varepsilon_{tol},k_{max}]$.
\item If $\ell_1<\ell_2$, we set $R_0^-=\ell_1$, $R_0^+=\ell_2$. If $\gamma_2=1$, then any $(\xi_1,a_1,\xi_2,a_2)\in C_1^0$ such that $(\xi_1,a_1)\in \Sigma_1$ and $(\xi_2,a_2)\in S(X_2,u_2,c_2,\ell)$ is a Nash equilibrium, for every $\ell\in[R_0^-,R_0^+]$. If $\gamma_2\neq 1$, the PPA game has no Nash equilibria.
\item If $\ell_1>\ell_2$, we set $R_0^-=\ell_2$, $R_0^+=\ell_1$. If $\gamma_1=1$, then any $(\xi_1,a_1,\xi_2,a_2)\in C_1^0$ such that $(\xi_1,a_1)\in S(X_1,u_1,c_1,\ell)$ and $(\xi_2,a_2)\in\Sigma_2$ is a Nash equilibrium, for every $\ell\in[R_0^-,R_0^+]$. If $\gamma_1\neq 1$, the PPA game has no Nash equilibria.
\item If $\ell_1=\ell_2=\ell$, any $(\xi_1,a_1,\xi_2,a_2)\in C_1^0$ such that $(\xi_1,a_1)\in S(X_1,u_1,c_1,\ell)$ and $(\xi_2,a_2)\in S(X_2,u_2,c_2,\ell)$ is a Nash equilibrium.
\end{enumerate}

The attentive reader will notice that our algorithm requires knowledge of the value functions of the Principals, which means solving standard Principal-Agent problems multiple times. In general, as discussed in \Cref{sec:intro}, solving the problem for a single value of the reservation utility is not an easy task. Nonetheless, we point out that the broad class of rational problems, in which the agent's expected utility is a rational function of the effort, can be tackled using the polynomial approach developed by \cite{renner15principal}. Indeed, in \Cref{sec:numerics}, we use this methodology to numerically solve several PPA games. 

In the particular case of an affine utility function, the Principal's complete value function can be obtained by solving the problem only once, for any reservation utility. We discuss this case in the next subsection.

\begin{rem}
Let us comment on the particular case when only one value function changes sign on $]-\infty,\bar R]$, say $V(X_{-i},u_{-i},c_{-i},\cdot)$. Note that the value $\ell_i$ in the routine above is not well defined, however there are two straightforward cases. If $V(X_{i},u_{i},c_{i},\cdot)$ remains positive, then $R_0^-=R_{-i}$, $R_0^+=\bar{R}$ and Nash equilibria exist only if $\gamma_i=1$. If the function remains negative, then $R_0^-=-\infty$, $R_0^+=R_{-i}$ and Nash equilibria exist only if $\gamma_{-i}=1$.
\end{rem}

\subsection{The case of affine utility functions}

In this subsection, we assume that the Agent's utility functions are affine, \emph{i.e.}, for every $i\in\{1,2\}$, $u_i(x)= \alpha_ix+\beta_i$ for some $\alpha_i>0$ and $\beta_i\in\R$. We also assume that the sets $\Xi_1$ and $\Xi_2$ are stable under translation by constants. In this setting we have $\bar{R_1}=\bar{R_2}=+\infty$ and sufficient conditions for \Cref{ass:existence-nash-competition} to hold are provided in \Cref{sec:affine}. Moreover, \Cref{lemma:linear-case-values} connects the values of the single Principal-Agent problem with different reservation utilities. More precisely, we have $V(X_i,u_i,c_i,R+\delta)=V(X_i,u_i,c_i,R)-\frac{\delta}{\alpha_i}$ for any $\delta\in\R$ and their solutions are linked via
$$
(\xi_i,a_i)\in S(X_i,u_i,c_i,R)\iff (u_i^{-1}(u_i(\xi_i)+\delta),a_i)=\left(\xi_i+\frac{\delta}{\alpha_i},a_i\right)\in S(X_i,u_i,c_i,R+\delta).
$$
Then, we can find a possible Nash equilibria for the PPA game by following the next steps
\begin{enumerate}
    \item For $i\in\{1,2\}$, find $V(X_i,u_i,c_i,0)$, then $V(X_i,u_i,c_i,R)=V(X_i,u_i,c_i,0)-\frac{R}{\alpha_i}$.
    \item Define $\tilde{R}_i:=\alpha_iV(X_i,u_i,c_i,0)$ for $i\in\{1,2\}$, then we will have that
\[
\big(V(X_i,u_i,c_i,R) = 0 \iff R=\tilde{R}_i \big) ~\wedge~ \text{sign}(V(X_i,u_i,c_i,R)) = \text{sign}(\tilde{R}_i-R).
\] 
\item If $\tilde{R}_1>\tilde{R}_2$ and $\gamma_1\neq 1$ or $\tilde{R}_2>\tilde{R}_1$ and $\gamma_2\neq 1$, there is no Nash equilibrium. Otherwise, let $R\in[\min\{\tilde{R}_1,\tilde{R}_2\},\max\{\tilde{R}_1,\tilde{R}_2\}]$.

\item Select $(\xi_1,a_1)\in S(X_1,u_1,c_1,R)$ and $(\xi_2,a_2)\in S(X_2,u_2,c_2,R)$.
\end{enumerate}
After the previous procedure, we have that $(\xi_1,a_1)$, $(\xi_2,a_2)$ is a Nash equilibrium of the PPA game. To find all possible Nash equilibria, we can relax Step (4) above. Namely, we only require:
\begin{enumerate}
\item[(4')] Select $(\xi_i,a_i)\in S(X_i,u_i,c_i,R)$ when $\tilde{R}_i\ge \tilde{R}_{-i}$ and $\gamma_i>0$. If $\tilde{R}_i< \tilde{R}_{-i}$ or $\gamma_i= 0$, select any $(\xi_i,a_i)\in\Sigma_i$ 
such that $(\xi_1,a_1,\xi_2,a_2)\in C_1^0$.
\end{enumerate}

\section{Numerical results} \label{sec:numerics}

Based on the discussion in \Cref{subsec:algorithm}, we present several illustrative examples in which we identify Nash equilibria of the PPA game using the routine presented in \Cref{subsec:algorithm}. We assume that the Principals' outcomes are discrete random variables. Hence, by \Cref{rem:when-assump-H-holds}, Condition \eqref{ass:H} is trivially satisfied.  In every example, we are certain that we find every Nash equilibria, since the necessary conditions established in \Cref{thrm:PPA-valuesign} cannot be satisfied at any other level of expected utility for the Agent.

\medskip
Inspired by \cite{renner15principal}, we assume that every action set is given by $A_i:=[0,1]$ and we let each outcome $X_i$ be a 3-valued discrete random variable, which follows a binomial distribution with probability $a$. That is, for any $i\in\{1,2\}$, for given $x_0^i<x_1^i<x_2^i$ we have 
\[
\mathbb{P}(X_i=x_j^i)=\binom{2}{j}a^j(1-a)^{2-j}, \quad \forall j \in\{0,1,2\}.
\]
In this setting, the set of contracts can be chosen simply as $\Xi_1=\Xi_2=\R^3$. We use the approach by \cite{renner15principal} to solve the individual Principal-Agent problems for any level of reservation utility. We are thus checking numerically, and \emph{a posteriori}, that \Cref{ass:existence-nash-competition} holds.

\medskip
From the discussion in \Cref{sec:differentiability}, the value function of each Principal is differentiable in all the examples. We thus use the two algorithms proposed in \Cref{subsec:algorithm} to determine the same set of Nash equilibria, and compare the performance of the algorithms at the end of the section.

\subsection{Exponential utilities} 
For each $i\in\{1,2\}$ we consider the utility functions $u_i:\R\rightarrow\R$ and cost functions $c_i:[0,1]\rightarrow\R$ of the Agent with the following form
\[
u_i(x)=1-e^{-\eta_i x}, \quad c_i(a)=\kappa_i a^2,
\]
with each $\eta_i>0$ and $\kappa_i>0$.

\medskip
From the Agent’s perspective, we consider two different job opportunities. One entails higher costs but offers a greater short-term reward, while the other involves lower costs and a higher long-term reward. 
We thus consider $\eta_1=1/3$, $\kappa_1=2$, $x_0^1=0$, $x_1^1=2$, $x_2^1=4$ and $\eta_2=1/2$, $\kappa_2=4$, $x_0^2=0$, $x_1^2=1$, $x_2^2=5$. 
We follow the steps discussed in \Cref{subsec:algorithm} to find Nash equilibria for this problem and present our results in \Cref{fig:ppa_exp_bisec}.

\begin{figure}[H]
    \centering
    \includegraphics[width=0.55\linewidth]{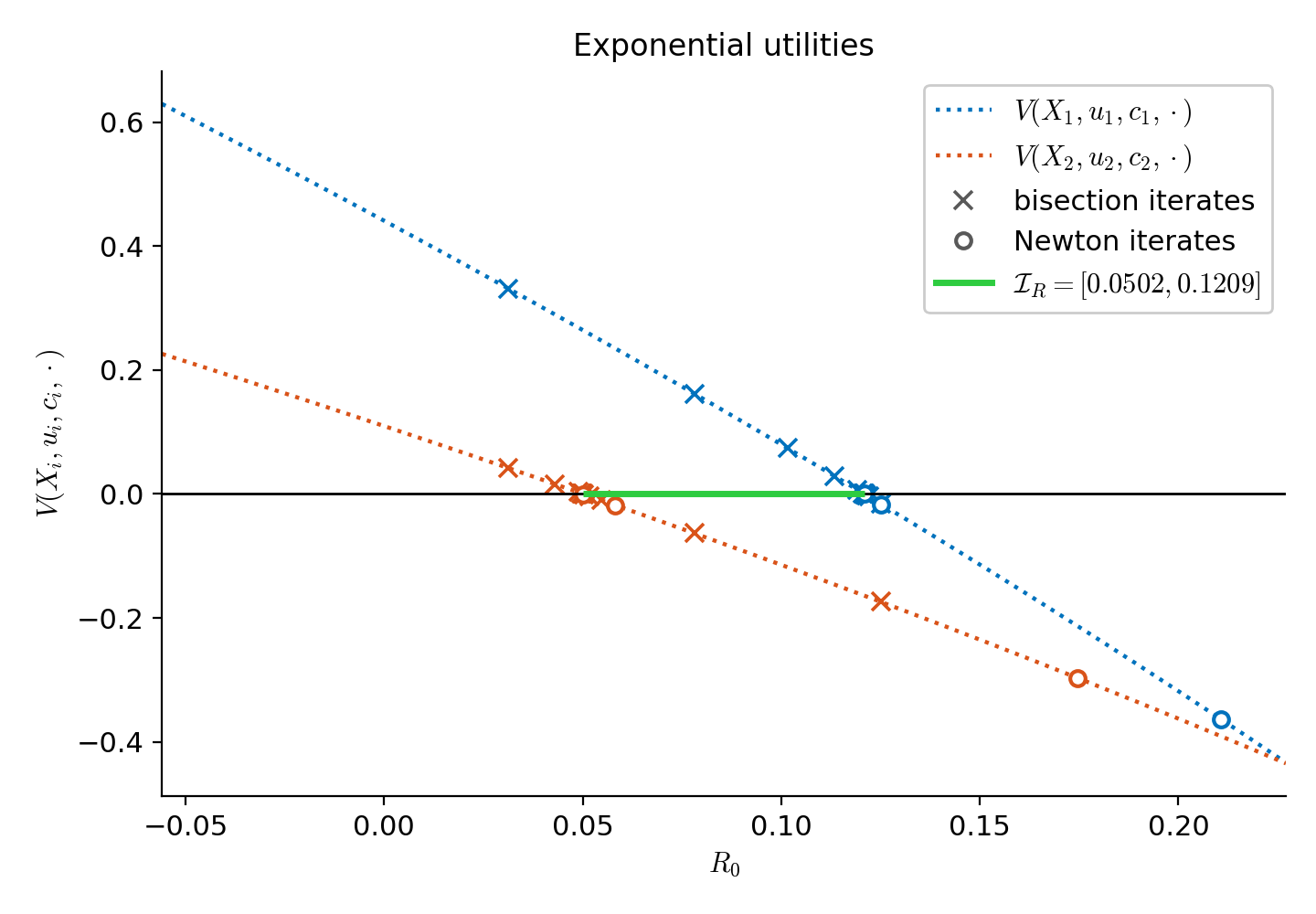}
    \caption{Computation of Nash equilibria via bisection and Newton's method.}
    \label{fig:ppa_exp_bisec}
\end{figure}

 The graph depicts the value functions of both Principals, when they can directly hire the Agent. In green, we can see the interval $\Ic_R=[0.0502,0.1209]$ where we land on cases (i) and (ii) of \Cref{prop:sufficient-condition-Nash}, where we can guarantee there exists Nash equilibria for the PPA game if $\gamma_1=1$ and $\gamma_2=0$. Indeed, from \Cref{rem:non-optimality-losing-principal}, we can construct Nash equilibria by picking $R_0\in\R$ in the interior of $\Ic_R$, a solution $(\xi_1,a_1)\in S(X_1,u_1,c_1,R_0)$ to Principal 1's contracting problem and any $(\xi_2,a_2)\in\Sigma_2$, such that $(\xi_1,a_1,\xi_2,a_2)\in C^0_1$. Any choice of such $(\xi_1,a_1),(\xi_2,a_2)$ constitutes a Nash equilibrium of the PPA game.

\subsection{Logarithmic utilities} In this example we consider the following form for the utility functions $u_i:\R_+\rightarrow\R$ and cost functions $c_i:[0,1]\rightarrow\R$ of the Agent
\[
u_i(x)=\ln({\eta_i x}+1), \quad c_i(a)=\kappa_i a^2,
\]
with each $\eta_i>0$ and $\kappa_i>0$. We point out that the utility functions have a restricted domain but the results of the paper can be easily extended to this case. We keep the same parameters and random variables $X_1$ and $X_2$ from the previous example, that is, $\eta_1=1/3$, $\kappa_1=2$, $x_0^1=0$, $x_1^1=2$, $x_2^1=4$ and $\eta_2=1/2$, $\kappa_2=4$, $x_0^2=0$, $x_1^2=1$, $x_2^2=5$. We 
present the results in \Cref{fig:ppa_log_bisec}.

\begin{figure}[H]
    \centering
    \includegraphics[width=0.55\linewidth]{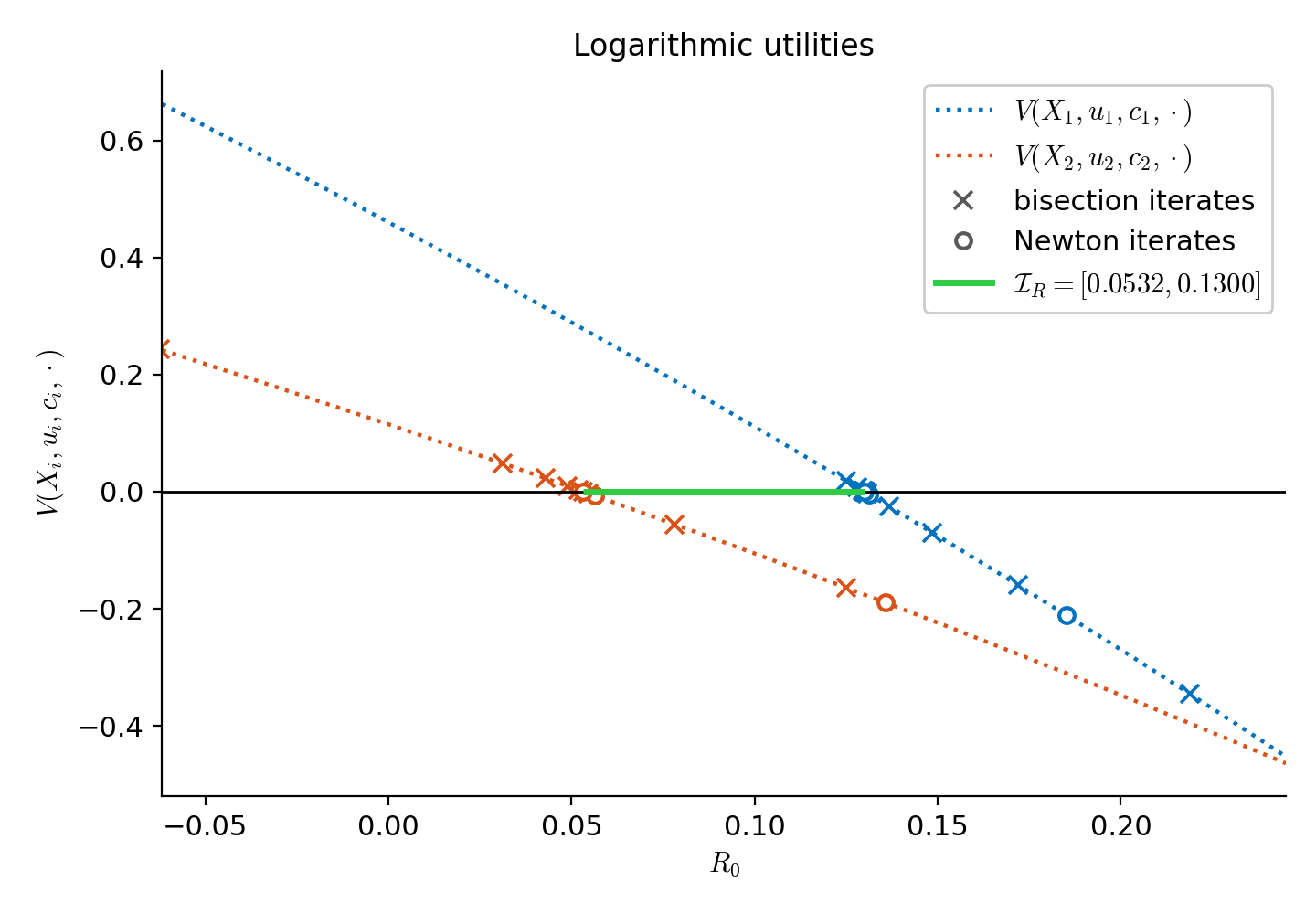}
    \caption{Computation of Nash equilibria via bisection and Newton's method.}
    \label{fig:ppa_log_bisec}
\end{figure}

The graph depicts the value functions of both Principals, when they can directly hire the Agent. Here, $\Ic_R=[0.0553,0.1300]$ and we can guarantee the existence of Nash equilibria for the PPA game if $\gamma_1=1$ and $\gamma_2=0$.


\subsection{Constant relative risk aversion (CRRA) utilities}
In this example we consider the set of contracts $\Xi_1=\Xi_2=\R_+^3$. For each $i\in\{1,2\}$ we consider the utility functions $u_i:\R_+\rightarrow\R$ and cost functions $c_i:[0,1]\rightarrow\R$ of the Agent with the following form
\[
u_i(x)=\frac{x^{1-\eta_i}}{1-\eta_i}, \quad c_i(a)=\kappa_i a^2,
\]
with each $\eta_i\in[0,1[$ and $\kappa_i>0$. We keep the same parameters and random variables $X_1$ and $X_2$ from the previous example, that is, $\eta_1=1/3$, $\kappa_1=2$, $x_0^1=0$, $x_1^1=2$, $x_2^1=4$ and $\eta_2=1/2$, $\kappa_2=4$, $x_0^2=0$, $x_1^2=1$, $x_2^2=5$. We follow the steps discussed in \Cref{subsec:algorithm} to find Nash equilibria for this problem and present our results in \Cref{fig:ppa_raiz_bisec}.

\begin{figure}[H]
    \centering
    \includegraphics[width=0.55\linewidth]{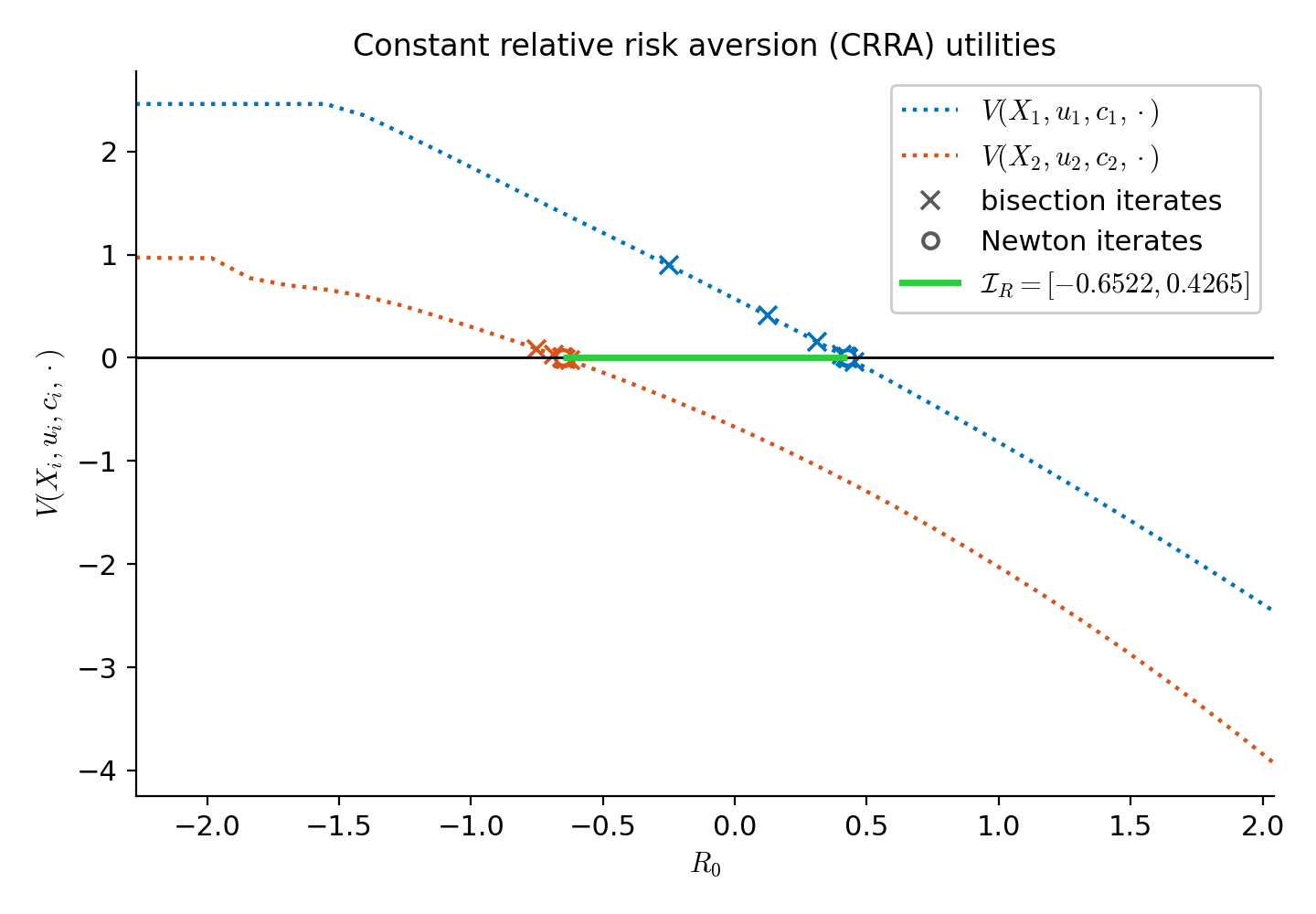}
    \caption{Computation of Nash equilibria via bisection and Newton's method.}
    \label{fig:ppa_raiz_bisec}
\end{figure}

The graph depicts the value functions of both Principals, when they can directly hire the Agent. Here, $\Ic_R=[-0.6519,0.4266]$ and we can guarantee the existence of Nash equilibria for the PPA game if $\gamma_1=1$ and $\gamma_2=0$.

\subsection{Comparison of the two algorithms}

In all the examples of the section, the Newton iterates reached objective values of order $10^{-4}$ in at most three evaluations, against
$\lceil\log_2((R-r)/10^{-4})\rceil\ge 13$ evaluations for bisection on the same brackets (see \Cref{tab:newton-vs-bisection}).
We point out that the accuracy of the zero is ultimately limited by the accuracy with which the individual problems are
solved: the SOS reformulation is nonconvex and is solved by a multistart local method, so that the computed values of
$V(X_i,u_i,c_i,\cdot)$ carry an error of order $10^{-5}$--$10^{-4}$ in the logarithmic and CRRA examples. Below that level the
sign of the value function is no longer reliable and neither bisection nor Newton can make progress, which is why the tolerance
$\varepsilon_{tol}$ should be chosen above the accuracy of the inner solver.

\begin{table}[H]
\centering
\begin{tabular}{lcccc}
\hline
Utility & $R_1$ & $R_2$ & N ($i=1$ / $i=2$) & B ($i=1$ / $i=2$)\\
\hline
Exponential & $0.12092$ & $0.05016$ & 3 / 3 & 14 / 14\\
Logarithmic & $0.12998$ & $0.05324$ & 3 / 3 & 14 / 14\\
CRRA & $0.42655$ & $-0.65219$ & 2 / 3 & 14 / 13\\
\hline
\end{tabular}
\caption{Zeros of $V(X_i,u_i,c_i,\cdot)$ in the examples of \Cref{sec:numerics} and number of evaluations needed by Function $N$ and by Function $B$ for achieving values of order $10^{-4}$.}
\label{tab:newton-vs-bisection}
\end{table}

\section{Conclusion} \label{sec:conclusion}
We have studied a competition model in which two Principals seek to hire the exclusive services of a common Agent. The main difficulty, compared with the standard PA problem, is that the Agent's reservation utility is endogenous and depends on the offers of both Principals. We have shown that this interaction can nevertheless be characterized through the value functions of the two associated PA problems. In particular, the equilibrium reservation utility must lie in an interval on which the Principals' individual value functions have different sign. This characterization also describes which Principal can profitably hire the Agent and how the equilibrium contracts relate to the solutions of the corresponding individual contracting problems. Furthermore, by incorporating the Agent's tie-breaking biases, we relate them to the signs of the Principals' payoff functions and formalize a general consistency property, without restricting \emph{a priori} the choices of an indifferent Agent.

Building on this characterization, we have proposed a procedure to identify Nash equilibria by solving the individual PA problems and comparing their value functions. Indeed, in any equilibrium, the Principal who gets to hire the Agent will do so by offering a solution to her individual contracting problem. We propose two methods for finding sign-change threshold for the value function of a Principal. The first one is bisection, which exploits monotonicity and can be applied even when the value function is discontinuous. The second one is Newton's method, valid under differentiability. We present numerical solutions to several cases of commonly used utility functions, which illustrate the applicability of the approach and show how the equilibrium can be computed through the proposed procedure.

Future research could naturally extend this competitive framework in several promising directions. One natural extension is to consider continuous-time settings or multi-agent environments where Agents can split their effort across multiple projects. Additionally, incorporating incomplete information regarding the Agent's risk preferences could further enrich the game-theoretic interactions between competing Principals.


\section*{Declarations}
\textbf{Competing interests.} The authors have not disclosed any competing interests.

\appendix
\section{Properties of the standard Principal-Agent problem} \label{sec:appendix-properties}

In this section we provide some properties of the standard Principal-Agent problem introduced in \eqref{e:PAmain-riskneutral}. We thus let $X$ be a random variable, $\Xi$ a set of $\Fc^X$-measurable random variables, $A\subset\R $ a compact set, $c:A\rightarrow\R$ a convex, differentiable and strictly increasing function and $u:\R\rightarrow\R$ a concave, differentiable and strictly increasing function. For any contract $\xi\in\Xi$, we define its upper-gap as 
\[
\alpha(\xi):=\sup\{\delta\in\R\,:\,u(\xi)+\delta\le \sup ({\rm im}\,u), ~\P-a.s.\},
\]
and the perturbed contracts 
\begin{equation}
\label{e:pertcontract}
(\forall \delta<\alpha(\xi))\quad 
\hat{\xi}_{\delta}:=u^{-1}(u(\xi)+\delta),\:\:\P-\text{a.s.},
\end{equation}
  which is well defined since $u^{-1}\colon\left]-\infty,\sup ({\rm im}\,u)\right[\to\R$ is well defined. Finally, we define the set
$$
\Sigma=\Big\{(\xi,a)\in\Xi\times A\,:\,a\in\argmax_{b\,\in\,A} \mathbb{E}^{b}[u(\xi)-c(b)] \Big\}.
$$

\begin{lem}\label{lemma:contratos-perturbados}
Let $(\xi,a)\in \Sigma$ and set $R:=\mathbb{E}^{a}[u(\xi)-c(a)]$. Then the following properties hold.
\begin{enumerate}[label=(\roman*)]
    \item \label{lemma:contratos-perturbadosi} For every $\delta<\alpha(\xi)$, the pair $(\hat{\xi}_{\delta},a)\in\Sigma$ and $\mathbb{E}^{a}[u(\hat{\xi}_{\delta})-c(a)]=R+\delta$.
    \item \label{lemma:contratos-perturbadosii} For every $\delta<\alpha(\xi)$, we have
$$\dfrac{1}{u'(\xi)}\delta\le \hat{\xi}_{\delta}-\xi\le \dfrac{1}{u'(\hat{\xi}_{\delta})}\delta, \quad \P-a.s.$$

\item \label{lemma:contratos-perturbadosiii}
Assume $\alpha(\xi)>0$. For every $b\in A$ and $\varepsilon>0$, there exists $\delta\in\left]0,\alpha(\xi)\right[$ such that
$$
(\forall \delta'\in ]0,\delta]) \quad 0<\mathbb{E}^{b}[\hat{\xi}_{\delta'}-\xi]\le\varepsilon.
$$
\end{enumerate}
\end{lem}

\begin{proof} \ref{lemma:contratos-perturbadosi}: Note that, for every $b\in A$, we have
\[
\mathbb{E}^b[u(\hat{\xi}_{\delta})-c(b)]=\mathbb{E}^b[u(\xi)+\delta-c(b)]=\mathbb{E}^b[u(\xi)-c(b)]+\delta.
\]
The result is clear since $(\xi,a)\in \Sigma$.

\ref{lemma:contratos-perturbadosii}: Since $u^{-1}$ is convex and $u$ is strictly increasing, the inverse function theorem yields
$$\xi+\frac{1}{u'(\xi)}\delta=u^{-1}(u(\xi))+(u^{-1})'(u(\xi))(u(\xi)+\delta-u(\xi))\le u^{-1}(u(\xi)+\delta)=\hat{\xi}_{\delta}$$
and, similarly,
$$\hat{\xi}_{\delta}-\frac{1}{u'(\hat{\xi}_{\delta})}\delta=u^{-1}(u(\xi)+\delta)+(u^{-1})'(u(\xi)+\delta)(u(\xi)-(u(\xi)+\delta))\le u^{-1}(u(\xi))=\xi.$$
Hence,
$$\frac{1}{u'(\xi)}\delta\le\hat{\xi}_{\delta}-\xi\le \frac{1}{u'(\hat{\xi}_{\delta})}\delta, \quad \P-a.s.,$$
and the result follows.

\ref{lemma:contratos-perturbadosiii}: Let us call $\alpha=\alpha(\xi)$ to simplify the notations. Let $b\in A$, let $\varepsilon>0$, let $\sigma\in\left]0,1\right[$ and set 
$$\delta=\min\left\{\sigma\alpha,\dfrac{\varepsilon}{\mathbb{E}^{b}\left[\frac{1}{u'(\hat{\xi}_{\sigma\alpha})}\right]}\right\}\in\left]0,\alpha\right[.$$ We have two cases.
\begin{itemize}
    \item If $\varepsilon<\sigma\alpha\mathbb{E}^{b}\left[\frac{1}{u'(\hat{\xi}_{\sigma\alpha})}\right]$, then $\delta=\varepsilon/\mathbb{E}^{b}\left[\frac{1}{u'(\hat{\xi}_{\sigma\alpha})}\right]<\sigma\alpha$. Moreover,
    noting that $\alpha(\hat{\xi}_{\delta})=\alpha-\delta>0$ yields $\sigma\alpha-\delta<\alpha(\hat{\xi}_{\delta})$, it follows from\footnote{Note that $u(\hat{\xi}_{\sigma\alpha})=u(\xi)+\delta+(\sigma\alpha-\delta)=u((\widehat{\hat{\xi}_{\delta}})_{\sigma\alpha-\delta})$} 
    $\hat{\xi}_{\sigma\alpha}=(\widehat{\hat{\xi}_{\delta}})_{\sigma\alpha-\delta}$, \ref{lemma:contratos-perturbadosi} and 
    \ref{lemma:contratos-perturbadosii} that
     $$
     0<\dfrac{1}{u'(\hat{\xi}_{\delta})}(\sigma\alpha-\delta)\le\hat{\xi}_{\sigma\alpha}-\hat{\xi}_{\delta},  \quad \P-a.s.
     $$
    Moreover, since $u$ is concave, $u'$ is decreasing, which implies that 
    $$
    \frac{1}{u'(\hat{\xi}_{\delta})}\le\frac{1}{u'(\hat{\xi}_{\sigma\alpha})},  \quad \P-a.s.
    $$
    Therefore, it follows from \ref{lemma:contratos-perturbadosii} that $0<\mathbb{E}^{b}[\hat{\xi}_{\delta}-\xi]\le \mathbb{E}^{b}\left[\frac{1}{u'(\hat{\xi}_{\delta})}\right]\delta\le \mathbb{E}^{b}\left[\frac{1}{u'(\hat{\xi}_{\sigma\alpha})}\right]\delta=\varepsilon$.
    \item If $\varepsilon\ge\sigma\alpha\mathbb{E}^{b}\left[\frac{1}{u'(\hat{\xi}_{\sigma\alpha})}\right]$, then $\delta=\sigma\alpha$ and the result is direct from \ref{lemma:contratos-perturbadosii}. 
\end{itemize}
The conclusion follows from the monotonicity of the contracts $\hat{\xi}_\delta$. \qedhere
\end{proof}

\begin{lem} \label{lema:matchingR0} 
Let $R_0\in\R$ be such that $S(X,u,c,R_0)\neq \emptyset$. Then for every $(\xi,a)\in S(X,u,c,R_0)$ we have $\mathbb{E}^a[u(\xi)-c(a)]=R_0$.
\end{lem}
\begin{proof} Suppose by contradiction $\mathbb{E}^a[u(\xi)-c(a)]>R_0$. Consider $0<\varepsilon\leq\mathbb{E}^a[u(\xi)-c(a)]-R_0$ and define the contract $\hat{\xi}=u^{-1}(u(\xi)-\varepsilon)$. By  \Cref{lemma:contratos-perturbados}
$$
a\in\argmax_{b\in A}\mathbb{E}^b[u(\hat{\xi})-c(b)],\quad \mathbb{E}^b[u(\hat{\xi})-c(b)]\geq R_0,\quad \mathbb{E}^{a}[\hat{\xi}-\xi]<0.
$$
By linearity of the expected value operator
$$
\mathbb{E}^a[X-\hat{\xi}]>\mathbb{E}^a[X-\hat{\xi}]+\mathbb{E}^a[\hat{\xi}-\xi]=\mathbb{E}^a[X-\xi],
$$
which is a contradiction.\qedhere
\end{proof}

\begin{lem}\label{lemma:decrecimiento-V}
The value function $V(X,u,c,\cdot):\R\to\R$ is non-increasing. If there exists an interval $I\subset\R$ such that $S(X,u,c,R)\neq\emptyset$ for every $R\in I$, then $V(X,u,c,\cdot):I\to\R$ is strictly decreasing. 
\end{lem}

\begin{proof} (i) For any $R\in\R$ define the feasible set
\begin{equation*}
F(R):=\left\{(\xi,a)\in\Sigma \mid \mathbb{E}^{a}[u(\xi)-c(a)]\geq R\right\}.    
\end{equation*}
Since for every $R_1<R_2$ it holds $F({R_2})\subseteq F({R_1})$, it follows that $V(X,u,c,\cdot):I\to\R$ is non-increasing. 

\medskip
(ii) Suppose there exist $R_1,R_2\in I$ with $R_1<R_2$ and such that $V(X,u,c,R_2)=V(X,u,c,R_1)$. Let $(\xi,a)\in S(X,u,c,R_2)$, then by \Cref{lemma:contratos-perturbados} the contract $\hat{\xi}=u^{-1}(u(\xi)-(R_2-R_1))$ satisfies
$$
a\in\argmax_{b\in A}\mathbb{E}^{b}[u(\hat\xi)-c(b)],\quad\mathbb{E}^{a}[u(\hat\xi)-c(a)]= R_1,\quad \mathbb{E}^{a}[\hat{\xi}-\xi]<0.
$$
Then, we would reach the following contradiction
$$
V(X,u,c,R_2) =\mathbb{E}^{a}[X-\xi] <\mathbb{E}^{a}[X-\xi]+\mathbb{E}^{a}[\xi-\hat{\xi}]=\mathbb{E}^{a}[X-\hat{\xi}] \leq V(X,u,c,R_1).
$$
We conclude that $V(X,u,c,\cdot)$ is a strictly decreasing function.
\end{proof}

\begin{lem}\label{lemma:continuidad-V}
Suppose there exists an open interval $I\subset\R,\,\bar{\alpha}>0$ and $K>0$ such that, for every $R\in I$, there exists $(\xi,a)\in S(X,u,c,R)$ satisfying $\alpha(\xi)>\bar{\alpha}$ and $\mathbb E^{a}\big[1/u'(\hat\xi_{\bar\alpha})\big]\le K$. Then $V(X,u,c,\cdot)$ is continuous on the set $I$.
\end{lem}

\begin{proof}
    \medskip
Let $R\in I$ and $\varepsilon>0$. Take $(\xi,a)\in S(X,u,c,R)$ such that $\alpha(\xi)>0$. From \Cref{lemma:contratos-perturbados}\ref{lemma:contratos-perturbadosiii}, take $\delta>0$ such that $R+\delta\in I$ and $\mathbb{E}^{a}[\hat{\xi}_\delta-\xi]<\varepsilon$. It follows from \Cref{lemma:contratos-perturbados}\ref{lemma:contratos-perturbadosi} that $\hat{\xi_\delta}$ satisfies
\[
a\in\argmax_{b\in A}\mathbb{E}^{b}[u(\hat{\xi_\delta})-c(b)], ~~\mathbb{E}^{a}[u(\hat{\xi_\delta})-c(a)]= R+\delta,
\]
and therefore
\[
V(X,u,c,R)= \mathbb{E}^{a}[X-\xi] \leq\mathbb{E}^{a}[X-\xi]+ V(X,u,c,R+\delta)-\mathbb{E}^{a}[X-\hat{\xi}_\delta]<\varepsilon + V(X,u,c,R+\delta).
\]

Consider now $\delta\in\left]0,\min\{\bar{\alpha} ,\frac{\varepsilon}{K}\}\right[$ be such that $R-\delta\in I$  and take $(\eta,\beta)\in S(X,u,c,R-\delta)$ such that $\alpha(\eta)>\bar{\alpha}$ and $\mathbb E^{\beta}\big[1/u'(\hat\eta_{\bar\alpha})\big]\le K$, where $\hat\eta_{\bar\alpha}$ is the contract $\eta$ perturbed by $\bar\alpha$, as defined in \eqref{e:pertcontract}. Then, from \Cref{lemma:contratos-perturbados}\ref{lemma:contratos-perturbadosi}, we have  $\mathbb{E}^{\beta}[X-\hat{\eta}_\delta]\leq V(X,u,c,R)$. Then, from \Cref{lemma:contratos-perturbados}\ref{lemma:contratos-perturbadosii} and the monotonicity of $u'$, we conclude
\[
V(X,u,c,R-\delta) - V(X,u,c,R)   \leq  \mathbb{E}^{\beta}[X-\eta]-\mathbb{E}^{\beta}[X-\hat{\eta}_\delta] \leq \delta \mathbb E^{\beta}\big[1/u'(\hat\eta_{\delta})\big] <  K\delta < \varepsilon.
\]
Since $V(X,u,c,\cdot)$ is monotonic, we have thus its left- and right-continuity. 
\end{proof}

\begin{rem}
We can see in the proof of \Cref{lemma:continuidad-V} that the right-continuity of the value function does not require the existence of $\bar{\alpha}$ and $K$, but rather the positivity of the upper-gaps. Note that if $X$ takes values in a finite set, then $\bar{\alpha}$ exists (see \Cref{rem:when-assump-H-holds}) and, since $u'$ is bounded, also exists $K$. 
\end{rem}

\section{Existence of solutions to the standard Principal-Agent problem} \label{sec:existence}

Based on \Cref{ass:existence-nash-competition}, we present examples of settings in which one can guarantee existence of solutions to the standard Principal-Agent problem. In the context of \Cref{sec:appendix-properties}, define the expected utility of the Principal $W:\Xi\times A\rightarrow\R$ and of the Agent $U:\Xi\times A\rightarrow\R$ respectively by $W(\xi,a):=\E^a[X(a)-\xi]$ and $U(\xi,a):=\E^a[u(\xi)-c(a)]$. Define also the feasible set of the Principal and the maximum utility of the Agent as follows, for $R\in\R$
\[
F(R) := \{ (\xi,a)\in\Sigma: ~ U(\xi,a) \geq R \}, \quad \bar{R}:= \sup_{(\xi,a)\in\Sigma} U(\xi,a).
\]

\subsection{Compact space of contracts}  \label{sec:example-compact}
Inspired by \cite{renner15principal}, suppose that the contract set $\Xi$ is a compact metric space and the expected utility functions $W$ and $U$ are continuous. Then \Cref{ass:existence-nash-competition} holds.

Indeed, consider the value function $\Psi:\Xi\rightarrow\R$ defined by $\Psi(\xi):=\max_{a\in A} U(\xi,a)$. It follows from Berge's maximum\footnote{See for instance see for instance Theorem 17.31 in \cite{aliprantis2006infinite}.}  theorem that $\Psi$ is continuous. Note that, by compactness, the values $\bar{R}$ is attained. Then, it is easy to see that for $R\le\bar{R}$ the set $F(R)$ is non-empty and we can represent it as
\[
F(R) = \{ (\xi,a)\in\Xi\times A: U(\xi,a) = \Psi(\xi), ~U(\xi,a)\geq R \}.
\]
Moreover, by continuity of $U$ and $\Psi$, the set $F(R)$ is a closed subset of the compact $\Xi\times A$ and therefore compact. Since $W$ is continuous, we conclude that $S(X,u,c,R)\neq\emptyset$ for every $R\le\bar{R}$. 

In infinite dimension, we can choose 
the set $\Xi$ as a classical compact subsets of $L^p$, 
while continuity of functions $U$ and $W$ hold under dominated convergence assumption on the densities of $X$. In the case of finite outcomes,
compactness follows from closedness and boundedness of $\Xi$, while continuity of functions $U$ and $W$ hold under continuity of the density of $X$.

\subsection{Outcome with finite support} \label{sec:example-finite-support}
Suppose that $X$ takes finitely many values, denoted by $x_0<x_1<\dots<x_n$. In this case, we can identify a contract $\xi$ with a vector $w=(w_0,\dots,w_n)\in\R^{n+1}$. Suppose that $\Xi$ is a closed subset of $\R^{n+1}$ which contains the deterministic contracts, that is, with the form $\{(w,\dots,w): w\in\R\}$. Suppose moreover that the probabilities $a\mapsto p_j(a):= \P^a(X=x_j)$ are continuous on $A$ and there exists $\underline{p}>0$ such that $p_j(a)\geq \underline{p}$ for every $j\in\{0,\dots,n\}$. Finally, we assume $M_X:=\sup_{a\in A} \E^a[X]<\infty$ and that the utility function $u$ is not affine\footnote{The affine case is discussed in \Cref{sec:affine}}, that is
\[
\theta_0 := \inf_{w\in\R} u'(w) < \theta_1:= \sup_{w\in\R} u'(w).
\]

Define 
$\underline{c}:=\min_{a\in A} c(a)$. Some simple computations, and Jensen's inequality, show that every $(\xi,a)$ in the feasible set $F(R)$ satisfies $\E^a[\xi]\geq u^{-1}(R+\underline{c})$. Moreover, by Ces\`aro's lemma, for every $\theta\in(\theta_0,\theta_1)$ the map $\Phi_\theta(x):=\theta x-u(x)$ is coercive and hence the set
\[
G_\lambda:= \{ (\xi,a)\in F(R): W(\xi,a) \geq \lambda\},
\]
is bounded for any $\lambda\in\R$. Let $\lambda_0$ be such that $G_{\lambda_0}\neq\emptyset$, otherwise the problem is trivial. Since, by continuity, the set $G_\lambda$ is closed, the compactness of $G_{\lambda_0}$ implies that $S(X,u,c,R)\neq\emptyset$ for every $R< \bar{R}$.

\section{The case of affine utility of the Agent} \label{sec:affine}

In this section, we provide some results for the case in which the Agent's utility function is affine, that is, it is given by $u(x):=\alpha x+\beta$, for some $\alpha,\beta\in\R$ with $\alpha>0$. Notice that in this case \eqref{ass:H} holds.

Throughout this section, we assume that the set of admissible contracts $\Xi$ is stable under translations by constants, that is, $\xi\in\Xi$ implies $\xi+\delta\in\Xi$ for every $\delta\in\R$. This ensures that the transformations of contracts used in the next proof are admissible. We also assume the existence of $R_0$ such that the set $S(X,u,c,R_0)\neq\emptyset$.

\begin{rem}\label{rem:existencia-R0}
A sufficient condition for the set $S(X,u,c,R_0)$ being nonempty is that the set of implementable actions $A_I:=\{a\in A: \exists\xi\in\Xi, (\xi,a)\in\Sigma \}$ is compact and $a\mapsto \E^a[X]$ is upper semi-continuous. See \Cref{lemma:existencia-R0} for the details.      
\end{rem}

\begin{lem}\label{lemma:linear-case-values}
In the context of this section
$$
(\xi,a)\in S(X,u,c,R_0)\implies \left(\xi+\frac{\delta}{\alpha},a\right)\in S(X,u,c,R_0+\delta),\,\forall\delta\in\R.
$$
Consequently, for every $\delta\in\R$ we have $V(X,u,c,R_0+\delta)=V(X,u,c,R_0)-\frac{\delta}{\alpha}$.
\end{lem}
\begin{proof} Let $(\xi,a)\in S(X,u,c,R_0)$ and $\delta\in \R$. From \Cref{lemma:contratos-perturbados}\ref{lemma:contratos-perturbadosi}, we have that $(\hat\xi_\delta,a)$ satisfies the constraints of the problem $PA(X,u,c,R_0+\delta)$, so we just have to show that it is a solution to the problem. Suppose, to the contrary, that there exists $(\eta,e)$ feasible in $PA(X,u,c,R_0+\delta)$ such that
$$
\E^{e}[X-\eta]>\E^{a}[X-\hat\xi_\delta].
$$
Note that, for any $x\in\R$
$$
u^{-1}(u(x)+\delta)=\frac{(\alpha x+\beta+\delta)-\beta}{\alpha}=x+\frac{\delta}{\alpha}.
$$
Hence $\hat\xi_\delta=\xi+\frac{\delta}{\alpha}$, which implies
$$
\mathbb{E}^a[X-\hat\xi_\delta]=\mathbb{E}^{a}[X-\xi]-\frac{\delta}{\alpha}.
$$
Now, again from \Cref{lemma:contratos-perturbados}\ref{lemma:contratos-perturbadosi}, the pair $(\hat\eta_{-\delta},e)$ is feasible in $PA(X,u,c,R_0)$, and
$$ \mathbb{E}^{e}[X-\hat\eta_{-\delta}] =\mathbb{E}^{e}[X-\eta]+\frac{\delta}{\alpha}
>\E^{a}[X-\hat\xi_\delta]+\frac{\delta}{\alpha} =\mathbb{E}^{a}[X-\xi].
$$
This would imply $(\xi,a)\not\in S(X,u,c,R_0)$, which is a contradiction. \qedhere
\end{proof}

\begin{lem}\label{lemma:existencia-R0}
Under the assumptions in \Cref{rem:existencia-R0}, for every $R_0\in\R$ the set $S(X,u,c,R_0)$ is nonempty.
\end{lem}

\begin{proof}
For every $a\in A_I$, choose $\xi^a\in\Xi$ such that $(\xi^a,a)\in\Sigma$. Define now the transformation
\[
\tilde\xi^a := \xi^a+ \frac{R_0-\E^a[u(\xi^a)-c(a)]}{\alpha},
\]
and notice that the contract $\tilde\xi^a$ satisfies
\[
(\tilde\xi^a,a)\in\Sigma, \quad  \E^a[u(\tilde\xi^a)-c(a)]=R_0.
\]
Moreover, some simple computations lead to
\[
\E^a[X-\tilde\xi^a] = \frac{\beta-R_0}{\alpha} + \E^a[X]-\frac{c(a)}{\alpha} := J_{R_0}(a)
\]
Since $J_{R_0}$ is upper semicontinuous on the compact set $A_I$, there exists $a^\ast\in A_I$ such that
\[
J_{R_0}(a^\ast)
=
\max_{a\in A_I}J_{R_0}(a).
\]
Finally, note that $(\tilde\xi^{a^\ast},a^\ast)$ is feasible for
Principal-Agent problem at reservation utility $R_0$ and, for every feasible pair $(\xi,a)$,
\[
\E^a[X-\xi]
\leq
J_{R_0}(a)
\leq
J_{R_0}(a^\ast)
=
\E^{a^\ast}[X-\tilde\xi^{a^\ast}].
\]
Thus $(\tilde\xi^{a^\ast},a^\ast)\in S(X,u,c,R_0)$.
\end{proof}

\section{On the differentiability of the value function}
\label{sec:differentiability}
In this section we study the differentiability of the value function of the standard Principal-Agent problem. We assume that $S(X,u,c,R)\neq\emptyset$ for every $R< \bar{R}$, as it happens in the examples discussed in \Cref{sec:existence}. Moreover, for $R<\bar{R}$, we define the quantities
\[
\ell_1(R):= \inf_{(\xi,a)\in S(X,u,c,R)} \E^a\bigg[\frac{1}{u'(\xi)}\bigg], \quad \ell_2(R):= \sup_{(\xi,a)\in S(X,u,c,R)} \E^a\bigg[\frac{1}{u'(\xi)}\bigg].
\]
We start by studying the directional derivatives of the value function $V(X,u,c,R)$ with respect to the reservation utility of the Agent. We use thus the simplified notation $D^+V(R), D_+V(R)$, $D^-V(R)$ and $D_-V(R)$ for the Dini derivatives of $V$, as well as $V_+^\prime$, $V_-^\prime$ and $V^\prime$ for the directional derivatives and the derivative of $V$.

\begin{prop} \label{prop:rightderivative1}
Let $R<\bar{R}$ be such that there exists $(\xi,a)\in S(X,u,c,R)$ satisfying $\alpha(\xi)>0$ and $\mathbb E^{a}\big[1/u'(\hat\xi_{\bar\delta})\big]<+\infty$ for some $\bar\delta\in \left]0,\alpha(\xi)\right[$. Then $D_+V(R) \geq - \ell_1(R)$.
\end{prop}

\begin{proof}
Let $R$ and $(\xi,a)$ be as in the statement and let $\delta\in\left]0,\bar\delta\right[$. By Lemma~\ref{lemma:contratos-perturbados}\ref{lemma:contratos-perturbadosi}
\[
\mathbb{E}^a[X-\hat\xi_\delta] \leq V(X,u,c,R+\delta).
\]
Therefore, using Lemma~\ref{lemma:contratos-perturbados}\ref{lemma:contratos-perturbadosii}, we obtain
$$
-\mathbb{E}^a\left[ \frac{1}{u'(\hat\xi_\delta)} \right]
\leq -\mathbb{E}^a\left[ \frac{\hat\xi_\delta-\xi}{\delta} \right]
\leq \frac{V(X,u,c,R+\delta)-V(X,u,c,R)}{\delta}.
$$
Note that, by monotonicity of $u'$, we have $\mathbb E^{a}\big[1/u'(\hat\xi_{\delta})\big]< \mathbb E^{a}\big[1/u'(\hat\xi_{\bar\delta})\big]<+\infty$. Then, taking the liminf as $\delta\to 0^+$, by monotone convergence we have
\[
 -\mathbb{E}^a\left[\frac{1}{u'(\xi)}\right] \leq D_+V(R).
\]
Taking supremum over $(\xi,a)\in S(X,u,c,R)$, we conclude $-\ell_1(R) \leq  D_+V(R)$.
\end{proof}

\begin{prop} \label{prop:rightderivative2}
Let  $R<\bar{R}$ be such that there exists $\hat\delta>0$ and $(\eta,\beta)\in S(X,u,c,R)$ satisfying
\[
(\forall \delta \in\,]0,\hat\delta\,[) (\exists~ (\eta_\delta,\beta_\delta)\in S(X,u,c,R+\delta) ) \quad\limsup_{\delta\to 0^+} \E^{\beta_\delta}\bigg[ \frac{1}{u'(\tilde\eta_{-\delta})}\bigg]=\E^\beta\bigg[ \frac{1}{u'(\eta)}\bigg]<+\infty,
\]
where we denote $\tilde\eta_{-\delta} = u^{-1}(u(\eta_\delta)-\delta)$. Then $D^+V(R) \leq - \ell_1(R)$.
\end{prop}

\begin{proof}
Note that, by \Cref{lemma:contratos-perturbados}\ref{lemma:contratos-perturbadosi}, we have $\mathbb{E}^{\beta_\delta}[X-\tilde\eta_{-\delta}] \leq V(X,u,c,R)$. Therefore, using \Cref{lemma:contratos-perturbados}\ref{lemma:contratos-perturbadosii}, we have
$$
\frac{V(X,u,c,R+\delta)-V(X,u,c,R)}{\delta}
\leq \mathbb{E}^{\beta_\delta}\left[ \frac{\tilde\eta_{-\delta}-\eta_\delta}{\delta} \right]
\leq -\mathbb{E}^{\beta_\delta}\left[ \frac{1}{u'(\tilde\eta_{-\delta})} \right].
$$
Taking the limsup as $\delta\to 0^+$, we have
$$
D^+V(R) \leq -\mathbb{E}^\beta\left[\frac{1}{u'(\eta)}\right] \leq -\ell_1(R). \qedhere
$$  
\end{proof}

\begin{rem}
We discuss below the conditions under which the technical assumptions in \Cref{prop:rightderivative2} are formally satisfied.
\begin{enumerate}
\item If the solution mapping $R \mapsto S(X, u, c, R)$ is upper hemicontinuous and takes compact values, then any sequence $(\eta_\delta, \beta_\delta) \in S(X, u, c, R+\delta)$ admits a convergent subsequence in $S(X, u, c, R)$, which is the first part of the assumption.

\item If the feasible mapping $R \mapsto F(R)$ is continuous and takes compact values, then Berge's maximum theorem guarantees that the solution mapping is upper hemicontinuous and takes compact values. Note that in both examples in \Cref{sec:existence} we proved the compactness of $F(R)$. Regarding the upper hemicontinuity, it holds whenever $\Xi$ is a compact metric space, as in Example \ref{sec:example-compact}. It also holds in Example \ref{sec:example-finite-support} due to closed graph property and finite dimensionality. Regarding the lower hemicontinuity, it holds if every $(\xi,a)\in F(R)$ satisfies $\alpha(\xi)>0$, since we can approximate such contracts via the perturbed contracts introduced in \Cref{sec:appendix-properties}. The upper-gap condition is discussed in \Cref{rem:when-assump-H-holds}.

\item Finally, the convergence
\[
\lim_{\delta\to 0^+} \E^{\beta_\delta}\bigg[ \frac{1}{u'(\tilde\eta_{-\delta})}\bigg]=\E^\beta\bigg[ \frac{1}{u'(\eta)}\bigg],
\]
holds in Example \ref{sec:example-finite-support} due to the continuity of the mappings $p_j$ and $u'$, and it also holds in Example \ref{sec:example-compact} if the density of $X$ is continuous in $L^1$ with respect to $a$ and if the integrand $\frac{1}{u'(\tilde\eta_{-\delta})}$ is uniformly bounded.
\end{enumerate}
\end{rem}

\begin{cor}\label{cor:right-derivative}
Let $R<\bar{R}$ satisfy the assumptions of \Cref{prop:rightderivative1} and \Cref{prop:rightderivative2}. Then the right-sided derivative of $V$ at $R$ exists and satisfies
\[
V'_+(R) = D^+V(R) = D_+V(R) = -\ell_1(R).
\]
\end{cor}

We move now to the left derivatives of $V$. Analogously to \Cref{prop:rightderivative1} and \Cref{prop:rightderivative2}, we can prove the following results. 

\begin{prop} \label{prop:leftderivative1}
Let $R<\bar{R}$ be such that there exists $\hat\delta>0$ and $(\eta,\beta)\in S(X,u,c,R)$ satisfying
\[
(\forall \delta \in \,]0,\hat\delta\,[) (\exists~ (\eta_{-\delta},\beta_{-\delta})\in S(X,u,c,R-\delta) ) \quad \lim_{\delta\to 0^+} \E^{\beta_{-\delta}}\bigg[ \frac{1}{u'(\tilde\eta_\delta)}\bigg]=\E^\beta\bigg[ \frac{1}{u'(\eta)}\bigg]<+\infty,
\]
where $\tilde\eta_\delta = u^{-1}(u(\eta_{-\delta})+\delta)$. Then $D_-V(R) \geq - \ell_2(R)$.
\end{prop}

\begin{proof}
Let $\delta \in (0,\hat\delta)$. By Lemma~\ref{lemma:contratos-perturbados}\ref{lemma:contratos-perturbadosi}, we have $\mathbb{E}^{\beta_{-\delta}}[X-\tilde\eta_\delta] \leq V(X,u,c,R)$. Therefore, using Lemma~\ref{lemma:contratos-perturbados}\ref{lemma:contratos-perturbadosii},
$$
\frac{V(X,u,c,R)-V(X,u,c,R-\delta)}{\delta} 
\geq \mathbb{E}^{\beta_{-\delta}}\left[ \frac{\eta_{-\delta}-\tilde\eta_\delta}{\delta} \right] 
\geq -\mathbb{E}^{\beta_{-\delta}}\left[ \frac{1}{u'(\tilde\eta_\delta)} \right].
$$
Taking the liminf as $\delta\to 0^+$, we obtain
$$
D_-V(R) \geq -\mathbb{E}^\beta\left[\frac{1}{u'(\eta)}\right] \geq -\ell_2(R). \qedhere
$$
\end{proof}

\begin{prop} \label{prop:leftderivative2}
Let $R<\bar{R}$ be such that some $(\xi,a)\in S(X,u,c,R)$ satisfies $\alpha(\xi)>0$ and $\mathbb E^{a}\big[1/u'(\xi)\big]<+\infty$. Then $D^-V(R) \leq - \ell_2(R)$.
\end{prop}

\begin{proof}
Let $(\xi,a)\in S(X,u,c,R)$ and $\delta>0$. By Lemma~\ref{lemma:contratos-perturbados}\ref{lemma:contratos-perturbadosi}, $\mathbb{E}^a[X-\hat\xi_{-\delta}] \leq V(X,u,c,R-\delta)$. Using Lemma~\ref{lemma:contratos-perturbados}\ref{lemma:contratos-perturbadosii}, we obtain
$$
\frac{V(X,u,c,R)-V(X,u,c,R-\delta)}{\delta} 
\leq \mathbb{E}^a\left[ \frac{\hat\xi_{-\delta}-\xi}{\delta} \right] 
\leq -\mathbb{E}^a\left[ \frac{1}{u'(\hat\xi_{-\delta})} \right].
$$
Taking the limsup as $\delta\to 0^+$, by the Monotone Convergence Theorem, we have
$$
D^-V(R) \leq -\mathbb{E}^a\left[\frac{1}{u'(\xi)}\right].
$$
Finally, taking the infimum over $(\xi,a)\in S(X,u,c,R)$ yields $D^-V(R) \leq - \ell_2(R)$.
\end{proof}

\begin{cor} \label{cor:left-derivative}
Let $R<\bar{R}$ satisfy the assumptions of Proposition~\ref{prop:leftderivative1} and Proposition~\ref{prop:leftderivative2}. Then the left-sided derivative of $V$ at $R$ exists and satisfies
\[
V'_-(R) = D^-V(R) = D_-V(R) = -\ell_2(R).
\]
\end{cor}

We conclude with the main result of this section, which characterizes the differentiability of the value function. In order to present it, we define the set $\tilde S(X,u,c,R)$ as the set of elements $(\xi,a)\in S(X,u,c,R)$ such that $\alpha(\xi)>0$ and $\mathbb E^{a}\big[1/u'(\hat\xi_{\bar\delta})\big]<+\infty$ for some $\bar\delta\in \,]0,\alpha(\xi)\,[$. We define finally, for $R<\bar{R}$, the quantities
\[
\tilde\ell_1(R):= \inf_{(\xi,a)\in \tilde S(X,u,c,R)} \E^a\bigg[\frac{1}{u'(\xi)}\bigg], \quad \tilde\ell_2(R):= \sup_{(\xi,a)\in \tilde S(X,u,c,R)} \E^a\bigg[\frac{1}{u'(\xi)}\bigg].
\]
It is clear that $\tilde\ell_1(R)\geq \ell_1(R)$ and $\tilde\ell_2(R)\leq \ell_2(R)$ for every $R$. If the function $u$ is surjective or if $X$ has finite support then we have $\tilde\ell_1(R)=\ell_1(R)$ and $\tilde\ell_2(R) =  \ell_2(R)$, see \Cref{rem:when-assump-H-holds}. In those cases, and any others where the two equalities hold, the next result characterizes the differentiability of the value function $V$.

\begin{prop} \label{prop:differentibility-iff}
Let $R<\bar{R}$ be such that the assumptions of Propositions \ref{prop:rightderivative1}, \ref{prop:rightderivative2} and \ref{prop:leftderivative1} are satisfied. If $\ell_1(R)=\ell_2(R)$ then the value function $V$ is differentiable at $R$. Conversely, if $V$ is differentiable at $R$ then $\tilde\ell_1(R)=\tilde\ell_2(R)$.
\end{prop}

\begin{proof} 
Note that the assumption in \Cref{prop:rightderivative1} is stronger than the one in \Cref{prop:leftderivative2}. If $\ell_1(R)=\ell_2(R)$, it follows from Corollaries \ref{cor:right-derivative} and \ref{cor:left-derivative} that $V$ is differentiable at $R$. 

Conversely, let $(\xi,a)\in S(X,u,c,R)$ be such that $\alpha(\xi)>0$ and $\mathbb E^{a}\big[1/u'(\hat\xi_{\bar\delta})\big]<\infty$ for some $\bar\delta\in (0,\alpha(\xi))$. It follows from \Cref{prop:rightderivative1} and \Cref{prop:leftderivative2}
\begin{equation*}\label{e:dini-V}
D^-V(R) 
\;\le\; -\,\E^{a}\!\left[\frac{1}{u'(\xi)}\right]
\;\le\; D_+V(R). 
\end{equation*}
Therefore, if $V$ is differentiable at $R$, then for every $(\xi,a)\in \tilde S(X,u,c,R)$ it holds 
\begin{equation} \label{e:derivada-V}
V'(R)=-\,\E^{a}\!\left[\frac{1}{u'(\xi)}\right],
\end{equation}
which implies $\tilde\ell_1(R)=\tilde\ell_2(R)$.
\end{proof}

\begin{rem}\label{rem:derivada-V}
\begin{enumerate}[label=(\roman*)]
\item As a by-product, if $V$ is differentiable at $R$, then every $(\xi,a)\in \tilde S(X,u,c,R)$ shares the same value of $\E^{a}[1/u'(\xi)]$.
\item In the affine case $u(x)=\alpha x+\beta$, \eqref{e:derivada-V} reduces to $V'\equiv-1/\alpha$,
which is \Cref{lemma:linear-case-values}.
\item Formula \eqref{e:derivada-V} is consistent with the classical characterization of the Lagrange multiplier of the
participation constraint in \cite{holmstrom1979moral}. When the first-order approach is valid and the solution is interior,
$1/u'(\xi(x))=\lambda+\mu\,\partial_a\log f(x,a)$, and taking $\E^{a}$ on both sides gives $\lambda=\E^{a}[1/u'(\xi)]$, so that
\eqref{e:derivada-V} is the envelope theorem $V'=-\lambda$. 
\end{enumerate}
\end{rem}


{\small
    \bibliographystyle{plainnat}
\bibliography{ref}
}

\end{document}